\documentclass[a4paper,10pt,english,reqno]{amsart}
\usepackage{graphicx, varioref, amscd ,color, bm, stmaryrd, mathrsfs, dsfont, yfonts, esint}
\usepackage[colorlinks=false]{hyperref}
\usepackage{amsmath}
\usepackage{tikz} %To make figure.
\usepackage{pgflibraryarrows} %Extra arrow heads
\usepackage{placeins} %Float placement
\usepackage[active]{srcltx}
\usepackage[mathscr]{eucal}

\newcommand{\comments}[1]{}
\newcommand{\R}{\mathbb{R}}
\newcommand{\N}{\mathbb{N}}

\theoremstyle{plain}

\theoremstyle{plain}
\newtheorem{theorem}{Theorem}[section]
\newtheorem{lemma}{Lemma}[section]
\newtheorem{proposition}[lemma]{Proposition}

\theoremstyle{definition}

\newtheorem{remark}[lemma]{Remark}

\numberwithin{equation}{section}
\allowdisplaybreaks[1]

\title[Error Estimates for Deep Learning]{Error Estimates for Deep 
Learning Methods in Fluid Dynamics}

\author[A. Biswas]{Animikh Biswas}
\address[Animikh Biswas]
{\newline Department of Mathematics and Statistics,
\newline University of Maryland Baltimore County,
\newline 1000 Hilltop Circle,
\newline  Baltimore, MD 21250, USA}
\email[]{abiswas@umbc.edu}
\urladdr{https://userpages.umbc.edu/~abiswas/}

\author[J. Tian]{Jing Tian}
\address[Jing Tian (corresponding author)]
{\newline Department of Mathematics,
	\newline Towson University,
	\newline 7800 York Road,
	\newline  Towson, MD 21252, USA}
\email[]{jtian@towson.edu}

\author[S. Ulusoy]{Suleyman Ulusoy}
\address[Suleyman Ulusoy]
{\newline Department of Mathematics and Natural Sciences,
\newline American University of Ras Al Khaimah,
\newline P.O. Box 10021,
\newline Ras Al Khaimah, UAE}
\email[]{suleyman.ulusoy@aurak.ac.ae}
\urladdr{https://aurak.ac.ae/en/dr-suleyman-ulusoy/}

\date{\today}


\begin{document}
	
	\begin{abstract}
		In this study, we provide  error estimates and stability  analysis of deep learning techniques for certain partial differential equations including the incompressible Navier-Stokes equations. In particular, we obtain explicit error estimates (in suitable norms) for the solution computed by optimizing a loss function in a Deep Neural Network (DNN) approximation of the solution, with a fixed complexity.
	\end{abstract}
	
	\maketitle
	%\allowdisplaybreaks
	
	%\tableofcontents
	
	\section{Introduction}
	Machine Learning, which has been at the forefront of the data science and artificial intelligence revolution in the last twenty years, has a wide range of applications in natural language processing, computer vision, speech and image recognition, among others \cite{Good, 37, 35}. Recently, its use has proliferated
	in computational sciences and physical modeling such as the modeling  of turbulence \cite{FSPS, TWPH, TDA, jimenez, WWX, WXP}.
	Moreover, machine learning methods (which are {\it mesh-free}) have gained wide applicability in obtaining numerical solutions  of various types of partial differential equations (PDEs); see \cite{Chrimachinelear, Bernermachinelear, GRK, RWTK, RPK, RK, SS, RWTK} and the references therein. The need for these studies stems from the fact that when using traditional numerical methods in a high-dimensional PDE, the methods sometimes become infeasible. High-dimensional PDEs appear in a number of models  for instance in the financial industry,  in a variety of contexts such as in derivative pricing models, credit valuation adjustment models, or portfolio optimization models. Such high-dimensional fully nonlinear PDEs are exceedingly difficult to solve as the computational effort for standard approximation methods grows exponentially with the dimension.
	For example, in finite difference methods, as the dimension of the PDEs increases, the number of grids increases considerably and there is a need for reduced time step-size. This increases the computational cost and memory demands. Under these circumstances, implementing the deep learning algorithms can be helpful. In particular, the neural networks approach in partial differential equations (PDEs) offer implicit regularization and can overcome the curse of high dimensions \cite{Chrimachinelear, Bernermachinelear}.

	Many infinite dimensional dynamical systems of practical interest arise in the context of geophysical flows related to the atmosphere and ocean. The Navier-Stokes and Euler equations, either alone or coupled with governing equations of other physical quantities such as the temperature and/or the magnetic field, are the fundamental equations governing the motion of fluids. They appear in the study of diverse physical phenomena such as aerodynamics, geophysics, atmospheric physics, meteorology and plasma physics. Especially, the Navier-Stokes equations can be used to model the Incompressible fluid flow and have been employed in describing many phenomena in science and engineering applications. For example, they are used in modeling the water flow in a pipe, air flow around a wing, ocean currents and weather. They are employed in the design of cars, aircrafts, and power stations, in the study of blood flow and many other applications. In this study, we mainly focus on using the neural networks techniques to solve two dimensional Navier-Stokes equations. However, we consider the elliptic case first to illustrate the fundamental issues involved.
	
 In literature, most studies \cite{GRK, RWTK, RPK, RK} focus on the numerical efficacy study of designing the neural network algorithms. 
 Concrete and complete mathematical analysis are meager for such methods applied to PDEs, in particular, for the Navier-Stokes equations, although some results on  convergence (as the complexity of the neural network tends to infinity) in the weak topology for some semilinear PDEs  can be found in \cite{SS}. The goal of this paper is to provide a mathematically rigorous error analysis of deep learning methods employed in \cite{GRK, RWTK, RPK, RK}
 for the general elliptic and two-dimensional Navier-Stokes equations. Our goal in this paper is not to analyze all the details of different possibilities of neural network architecture. Instead, we would like to provide a mathematically rigorous analysis  of the method, with {\it error estimates, and stability analysis}, similar in spirit to the probabilistic error analysis for machine learning algorithms for the Black-Scholes equations in \cite{Bernermachinelear}. We consider two different settings: the elliptic PDEs, mainly to fix ideas and illustrate our approach, and the  Navier-Stokes equations, which is the main focus of this work.
 Although our results are proven in the context of the two-dimensional Navier-Stokes equations, we note that our analysis applies equally well to the three dimensional case, up to the interval of existence of a {\it strong solution}, which in the two dimensional case, exists globally in time.  
 
 The computational algorithm  employed in machine learning of PDEs (for instance in \cite{GRK, RWTK, RPK, RK}) involves representing the approximate solution by a {\it Deep Neural Network (DNN)}, in  lieu of a spectral or finite element approximation, 
 and then minimizing, over all such representations, an appropriate  loss function, measuring the deviation of this representation from the PDE and the initial and boundary conditions. One important thing to note in this approach is the following. It is well-known that optimization of loss functions in a deep neural network is a non-convex optimization problem. Therefore, neither the existence nor the uniqueness of a global optimum is guaranteed. Nevertheless,  we side step this issue by obtaining  an explicit error estimate in terms of the attained value of the loss function (which takes the value zero for the true solution). The estimate we obtain in turn guarantees that the approximate solution  thus constructed converges, in the strong topology, to the true solution as the complexity of the networks tends to infinity.

 The rest of the paper is organized as follows. Section \ref{Sec:Preliminaries} provides the preliminaries for both Neural Network settings and approximation properties which will be used in this study. Section \ref{Sec:Main results} is devoted to the statement of our main results. In section \ref{Sec:elliptic}, we present the mathematical analysis of the neural network algorithm in the elliptic system. This also serves as a systematic introduction of our analysis. In section \ref{section:navier-stokes}, we present our main results in two dimensional Navier-Stokes equations. By using Hodge decomposition, we have shown that the approximate solution using the neural network algorithm is close to the actual solution of the two dimensional Navier-Stokes equations under certain conditions. Moreover, we have proved that our scheme is approximately stable. The existence of the approximate solution is shown by applying approximation properties of neural networks.
\section{Preliminaries}\label{Sec:Preliminaries}
	\subsection{Neural Networks}
%	Deep learning algorithms are used to analyze big data set %through statistics and mathematical models.
%	Neural networks are a class of function approximators employed %deep learning algorithms that have wide applications such as %natural language processing, gaming, medical diagnosis, etc.  It %has advantages in dealing with high-dimensional data and nonlinear %system.
In a DNN, we consider a mapping $f:{x}\mapsto  {y}$, where ${x}$ is the input variable and $ {y}$ is the output variable. The mapping function $f$ is obtained by (function) composition of {\it layer functions}, comprising  of an input layer, an output layer and multiple hidden layers, connected in  {\it neural network}. The details are as follows.

	In a DNN, each layer is a  function of the form
	$\sigma({w} {x}+b),\  {x} \in \R^d,  {w} = (w_1, \cdots, w_d),\ b \in \R.$ Here, $\sigma$ is called the \emph{activation function} and is usually taken to be either a sigmoid ($\sigma(x)=\frac{e^x}{e^x+1}$) , $\tanh$ or $\Re \ln$, where $\Re \ln (\xi) := \max(0, \xi)$. In applications to PDE, where we require adequate regularity of solutions, a popular choice is the $\tanh$ function where $\displaystyle \tanh(\xi) = \frac{e^{\xi}-e^{-\xi}}{e^{\xi}+e^{-\xi}}$.
	% the sigmoid (or the logistic function) is defined as $
	%%\displaystyle \sigma(\xi) = \frac{1}{2}\left(\tan h\left(\frac{\xi}%%{2}\right)+1\right).$  
	
	%\subsection{Multilayer Feedforward Network}
	
	Consider the collection of functions of the form
	\begin{equation}\label{i1}
		\sum \alpha_j f_1 \circ f_2 \circ f_3 \circ \cdots \circ f_{l_j}(x),
	\end{equation}
	where $f_i$ is a function of the form $\sigma(wx+b)$ described above.
%	 In practice, usually one type of activation function is chosen and fixed. We note that one can expand dimension in the middle of composition. The following fact is quite useful in practice: Feedforward networks are universal function approximations in $C^m(X),$ where $X$ is a compact subset of $\R^d.$ 
	 In \eqref{i1}, $\max{l_j}$ is called the depth of the network.
	 Henceforth, we will denote by $\mathscr{F}_N$ the class of functions in \eqref{i1}, where $N$ represents the network complexity (e.g. $N$ could be the sum of the ranks of the weight matrices $w$ and the number of layers in the DNN).

	For the sake of completeness, we give a schematic representation of a neural network. Here, we adapt the standard dense neural networks which can be expressed as a series of compositions:
	\begin{equation}\label{i2}
		\begin{split}
			y_2(x) &= \sigma(W_1x+b_1),\\
			y_3(y_2) &= \sigma(W_2y_2+b_2),\\
			&\cdot\\
			&\cdot\\
			&\cdot\\
			y_{n_l}(y_{n_l-1}) &= \sigma(W_{n_l-1}y_{n_l-1}+b_{n_l-1}),\\
			y_{n_l+1}(y_{n_l}) &=  \sigma(W_{n_l}y_{n_l}+ b_{n_l}),\\
			f_{\theta} &= y_{n_l+1}(y_{n_l}(\cdots(y_2(x)))),
		\end{split}
	\end{equation}
	where $\theta$ ensembles all the weights and  parameters.
	\begin{equation}\label{i3}
		\theta = \left\{ W_1, W_2, \cdots, W_{n_l}, b_1, \cdots, b_{n_l} \right\}.
	\end{equation}
	In practice, different neural network architectures are possible such as those involving  recurrent cells \cite{38}, convolutional layers \cite{35}, sparse convolutional neural networks \cite{36}, pooling layers, residual connections \cite{37}.

	In this study, we assume that our neural networks are equipped with uniformly bounded weights and the final bias term $b_{n_l}$. We do not need any boundedness assumption on the other bias terms $b_i$. 
\subsection{Function Approximation} 
Approximation properties of different DNNs has been studied extensively since the work of  Cybenko \cite{cgapprox} and Hornik \cite{hornik}; see
\cite{TC, xie2011errors} and the references therein for more recent work.
An important question in the approximation process is how many neural network layers are needed to guarantee the approxmation accuracy? In \cite{andrewapprox}, the author showed that by using the sigmoidal activation funciton, at most $O(\varepsilon^{-2})$ neurons are needed to achieve the order of approximation $\varepsilon$. In \cite{cgapprox}, Cybenko proved that continuous functions can be approximated with arbitrary precision by the DNNs with one internal layer and an arbitrary continuous sigmoidal function providing that no constraints are placed on the number of nodes or the size of the weights. Also, in \cite{hkappro}, Hornik {\it et. al.} provided the conditions ensuring that DNNs with a single hidden layer
and an appropriately smooth hidden layer activation function are capable of arbitrarily accurate approximation to an arbitrary function and its derivatives. A relevant theorem from \cite{xie2011errors}, relating the accuracy of the approximation of a DNN with  the complexity of the DNN and the regularity of the function being approximated, is given below.
%by increasing the number of hidden neurons \cite{xie2011errors}. 
	\begin{theorem}  \label{theorappr} Suppose that $\sigma \in C^{\infty} (\R)$, $\sigma^{(v)}(0)\neq 0$ for $\nu=0, 1, ...,$ and $K\subset \R^{d}$ is any compact set. If $f\in C^{k}(K)$, then a function $\phi_n$ represented by a DNN with complexity $n \in \N$
	exists such that
		\begin{align*}
			 \|D^{\alpha}f-D^{\alpha}\phi_n\|_{C(K)}=O\left(\frac{1}{n^{(k-|\alpha|)/d}}\omega\left(D^{\beta}f, \frac{1}{n^{1/d}}\right)\right)
		\end{align*}
		holds for all multi-indexes $\alpha$, $\beta$ with $|\alpha|\leq k$, $|\beta|=k$, where 
		\[
		\omega(g,\delta)
		= \sup_{x,y \in K, |x-y| \le \delta} |g(x)-g(y)|.
		\]
	\end{theorem}

	\section{Main results}\label{Sec:Main results}
	Let $\mathscr{F}_N$ be a DNN with complexity $N$, which is a finite dimensional function space on a bounded domain. Below is a list of our main results.
	\subsection{Elliptic case}
		Consider a bounded domain $\Omega$ of $\R^2$ and the following partial differential equation
		\begin{equation} \label{a1}
		\left\{
		\begin{split}
		\mathscr{L}u &= f, \\
		u|_{\partial \Omega } &= g,
		\end{split}
		\right.
		\end{equation}
		where  $\mathscr{L}: H^2(\Omega) \to L^2(\Omega)$ is a second order uniformly  bounded elliptic operator. 
In this study, for simplicity, we consider only the case $g=0$, although the general case is similar.

Recall that \eqref{a1} is well-posed and a unique  solution exists
satisfying
	$$M:=\|u\|_{H^2(\Omega)}\leq c\|f\|_{L^2(\Omega)}.$$
	Consequently, the minimization problem
	\begin{equation}\label{an3}
	\inf_{u \in\ \text{appropriate Sobolev class}}  \left\{ \left \|(\mathscr{L}u)(x) - f(x)\right \|_{L^2(\Omega)}^2 + \|u|_{\partial \Omega}\|_{L^2(\partial \Omega)}^2   \right\}
	\end{equation}
	has a unique solution, with the value of the infimum being 0, and the infimum is attained at the solution $u$ of \eqref{a1}. %Moreover, the infimum is obtained at a unique $u$. This follows from the well-posedness of \eqref{a1}.
	More generally, the same conclusion holds if we consider a loss function of the type
	\[
		L=\alpha^2 \|\mathscr{L}u-f\|_{L^2(\Omega)}^2 + 
		\beta^2\| u|_{\partial \Omega}\|_{L^{2}(\partial \Omega)}^2.
	\]
	
Thus, in order to approximate $u$ using a DNN, one considers the loss function	
	\begin{align}
		\label{j13}
		L=\alpha^2 \|\mathscr{L}u_N-f\|_{L^2(\Omega)}^2 + \beta^2\| u_N|_{\partial \Omega}\|_{L^{2}(\partial \Omega)}^2, 
		 u_N \in \mathscr{F}_N,
	\end{align}
	under the restriction that $\|u_N\|_{H^{2}(\Omega)}\leq \widetilde{M}$    (i.e. $\|u_N\|_{H^{2}(\Omega)}$ is bounded) for suitable $\widetilde{M}$ (e.g. $\widetilde{M}=2M$),
	with $\alpha, \beta>0$. Since the chosen activation function $\sigma=\tanh$ is smooth, in practice, this is achieved by restricting the (finite dimensional) parameter set in the neural network to a compact subset. 
	
	In the neural network framework, the optimization is usually conducted in a discrete setting as follows \cite{RPK}. More precisely, let $\mathscr{F}_N$ be a finite dimensional function space on a bounded domain $\Omega$. Choose a collocation points $\{x_j\}_{j=1}^m \subset \Omega$ and $\{y_j\}_{j=1}^n \subset \partial \Omega.$ Find
	\begin{equation}\label{a2}
	\inf_{u \in \mathscr{F}_N, \|u_N\|_{H^2(\Omega)} \le \widetilde{M}}  \left\{ \sum_{j=1}^m \alpha^2\left|(\mathscr{L}u)(x_j) - f(x_j)\right|^2 + \sum_{j=1}^n \beta^2\left|u(y_j)-g(y_j)\right|^2  \right \}.
	\end{equation}
	Note that \eqref{a2} may be regarded as a Monte Carlo approximation of the corresponding Lebesgue integrals. Consequently,
	for mathematical convenience, let us consider the following  optimization problem, namely, find
	\begin{equation}\label{a3}
	\inf_{u \in \mathscr{F}_N, \|u_N\|_{H^2(\Omega)} \le \widetilde{M}}  \left\{ \|(\mathscr{L}u)(x) - f(x)\|_{L^2(\Omega)}^2 + \|u|_{\partial \Omega}-g\|_{L^2(\partial \Omega)}^2   \right\}.
	\end{equation}
	The infimum can be attained provided that we restrict the parameters in $\mathscr{F}_N$ in a compact set. However in this case, the infimum may not be unique.
	\begin{remark} We can also use an unrestricted optimization
	in \eqref{a2} or \eqref{a3}. However, in this case, the condition on $\| u_N\|_{H^{2}(\Omega)} \le \widetilde{M}$
	can be replaced by suitably adding a penalty/regularization term in the loss function.   This converts the restricted minimization problem to an unrestricted one and is illustrated in the Navier-Stokes case. This drawback is due to the fact in contrast to spectral or finite element methods,  the boundary conditions are not encoded in a DNN, but rather are enforced ``approximately".
	\end{remark}
	
	In all the boundary integrals above, the quantity $u|_{\partial \Omega}$ is interpreted as trace in case $u \in H^1(\Omega)$. However, since $u_N$ is smooth, its trace coincides with its restriction on the boundary. Recall that the
	trace operator is defined as a bounded operator  $\gamma \in \mathscr{L}(H^1 (\Omega), L^2(\Gamma))$ such that $\gamma u$ is the restriction of $u$ to $\Gamma$ for every function $u\in H^1(\Omega)$ which is twice continuously differentiable in $\overline{\Omega}$.

	%Alternatively, we can also consider an unrestricted optimization with the loss funciton to be $\alpha^2 \|\mathscr{L}u_N-f\|_{L^2(\Omega)}^2 + \beta^2\| u_N|_{\partial \Omega}-g\|_{L^{2}(\partial \Omega)}^2$ and an extra condition: $\| u_N\|_{H^{2}(\Omega)}^2<M.$

		First, we show that when the approximate solution and actual solution are close to each other, we can control the loss function.
	\begin{theorem}\label{theorem1}
		Let $u$ be the solution of \eqref{a1} and $\alpha, \beta , \epsilon >0$. Then there exists $u_N\in \mathscr{F}_N$ with $\| u-u_N\|_{H^2(\Omega)}\leq\varepsilon$
		such that 
		\[
		\displaystyle \alpha^2 \|\mathscr{L}u_N-f\|_{L^2(\Omega)}^2 + \beta^2\| u_N|_{\partial \Omega}\|_{L^{2}(\partial \Omega)}^2\leq c\varepsilon^2\ \mbox{and}\ \|u_N\|_{H^{2}(\Omega)}\leq \widetilde{M},
		\]
		where $\widetilde{M}$ can be taken to be $2M=2\|u\|_{H^2(\Omega)}$.
	\end{theorem}
On the other hand, we can show that by controlling  the loss function, we can have a good approximation to the solution $u$ of \eqref{a1} by using a DNN. The requisite error estimate is given in the theorem below.
\begin{theorem}\label{thm2}
	Let $u$ be a solution of \eqref{a1} and $\alpha, \beta , \epsilon >0$. Assume that  $u_N\in \mathscr{F}_N$ is such that
	\begin{align}\label{j8}
	\alpha^2 \|\mathscr{L}u_N-f\|_{L^2(\Omega)}^2 + \beta^2\| u_N|_{\partial \Omega}\|_{L^{2}(\partial \Omega)}^2\leq \varepsilon^2,
	\end{align}	
	with $\| u_N\|_{H^{2}(\Omega)}\leq \widetilde{M}.$
	Then
	$$\|u-u_N\|_{H^1(\Omega)} \leq O((M+\widetilde{M})^{1/2}\varepsilon^{1/2}),$$
	and
	$$\|u-u_N\|_{L^2(\Omega)} \leq O((M+\widetilde{M})^{1/3}\varepsilon^{2/3}).$$
\end{theorem}
Observe that in the Theorem above, as expected, the error estimate in the $L^2$-norm is stronger than the error estimate in the $H^1$-norm. We show in the theorem below how the error estimate in the 
$L^2$-norm can be improved further by altering the loss function.
%improve Theorem \ref{thm2} by using a different approach.
	\begin{theorem}\label{thmay1}
	Let $u$ be a solution of \eqref{a1} and let $u_N\in \mathscr{F}_N$ be such that
	\begin{align}\label{j12}
	\alpha^2 \|\mathscr{L}u_N-f\|_{L^2(\Omega)}^2 + \beta^2\| u_N|_{\partial \Omega}\|_{H^{1/2}(\partial \Omega)}^2 \leq\varepsilon^2,
	\end{align}	
	with $\| u_N\|_{H^{2}(\Omega)}\leq \widetilde{M}.$
	Then
	$$\|u-u_N\|_{L^2(\Omega)} < O(\varepsilon).$$
\end{theorem}
		\subsection{Incompressible Navier-Stokes Equations}
		The incompressible Navier-Stokes equations (NSE) are given by
		\begin{equation}\label{c2}
		\begin{split}
	\partial_t u - \Delta u + u \cdot \nabla u + \nabla p &= f,\\
		\nabla \cdot u &= 0, \\
		u|_{\partial \Omega} &= 0,\\
		u(x, 0)&=u_0(x), x \in \Omega.
		\end{split}
		\end{equation}
	In \eqref{c2}, $u$ denotes the velocity of the fluid and $p$ the pressure.
		Similar to the elliptic case, we show that when applying the $\mathscr{F}_N$ on the Navier-Stokes equations, with a small loss function, the approximate solution and actual solution are close to each other.
		\begin{theorem}\label{thmay3}
		Assume that $u$ is a strong solution of the 2D NSE (\ref{c2}) and  $\tilde{u}_N \in \mathscr{F}_N$ such that
		%	\begin{equation}\label{c1}
		%	\begin{split}
		%	|| \tilde{u}_N (x, 0) - u_0(x)||_{L^2(\Omega)}^2 +|| \partial_t \tilde{u}_N - \Delta \tilde{u}_N + \tilde{u}_N\cdot \nabla \tilde{u}_N + \nabla \tilde{p}_N - f||_{L^2([0, T] \times \Omega)}^2 \\
		%	+ || \nabla \cdot \tilde{u}_N ||_{L^2([0, T] \times \Omega)}^2
		%	+ \lambda || \tilde{u}_N||_{L^4([0, T]; H^1(\Omega))}^4 \leq \varepsilon^2.
		%	\end{split}
		%	\end{equation}
		\begin{align}\label{cn1}
		%\begin{split}
		&\|\tilde{u}_N|_{\partial \Omega}\|_{L^4([0, T]; H^{1/2}(\partial \Omega))}^4
		+\|\tilde{u}_N (x, 0) - u_0(x)\|_{L^2(\Omega)}^2 \\ \nonumber
		&+\|\partial_t \tilde{u}_N - \Delta \tilde{u}_N + \tilde{u}_N\cdot \nabla \tilde{u}_N + \nabla \tilde{p}_N - f\|_{L^2(\Omega\times [0, T] )}^2 \\ \nonumber
		&+ \|\nabla \cdot \tilde{u}_N\|_{L^4([0, T]; L^2(\Omega))}^4
		+ \lambda\|\tilde{u}_N\|_{L^4([0, T]; H^1(\Omega))}^4 \leq \varepsilon^2.
		%	\end{split}
		\end{align}
		
		Then
		\begin{align}
		\label{diff1}
		\|u-\tilde{u}_N\|_{L^4([0, T]; L^2(\Omega))}\leq O \left(\varepsilon^{1/2}+\frac{\varepsilon}{{\lambda}^{1/4}}\right).
		\end{align}
	\end{theorem}
The reverse direction of Theorem \ref{thmay3} has also been proved.
	\begin{theorem}\label{thmay6}
	Given any $\varepsilon>0$, we can find $\tilde{u}_N \in \mathscr{F}_N$, such that
	\begin{align}
	\label{dec8}
	%\begin{split}
	&	\| \tilde{u}_N|_{\partial \Omega}\|_{L^4([0, T]; H^{1/2}(\partial \Omega))}^4+\| \tilde{u}_N (x, 0) - u_0(x)\|_{L^2(\Omega)}^2 \\ \nonumber
	&	+\| \partial_t \tilde{u}_N - \Delta \tilde{u}_N + \tilde{u}_N\cdot \nabla \tilde{u}_N + \nabla \tilde{p}_N - f\|_{L^2(\Omega\times [0, T] )}^2 \\ \nonumber
	&	+ \| \nabla \cdot \tilde{u}_N \|_{L^4([0, T] ; L^2(\Omega))}^4
	+ \lambda \| \tilde{u}_N\|_{L^4([0, T]; H^1(\Omega))}^4\leq O\left(\varepsilon^2\right).
	%\end{split}
	\end{align}
\end{theorem}
Furthermore, we prove that our scheme is approximately stable.
		\begin{theorem}\label{thmay4}
		Assume $\tilde u_{N_1} \in \mathscr{F}_{N_1}$ is the approximate solution of
		\begin{equation}\label{dec8n}
		\begin{split}
		\frac{\partial}{\partial t} u_1 - \Delta u_1 + u_1 \cdot \nabla u_1 + \nabla p_1 &= f_1, \\
		\nabla \cdot u_1 &= 0, \\
		u_1|_{\partial \Omega} &= 0, \\
		u_1 (x, 0)&=u_{0, 1} (x).
		\end{split}
		\end{equation}
		Assume $\tilde u_{N_2} \in \mathscr{F}_{N_2}$ is the approximate solution of
		\begin{equation}\label{dec9n}
		\begin{split}
		\frac{\partial}{\partial t} u_2 - \Delta u_2 + u_2 \cdot \nabla u_2 + \nabla p_2 &= f_2, \\
		\nabla \cdot u_2 &= 0, \\
		u_2|_{\partial \Omega} &= 0,\\
		u_2 (x, 0)&=u_{0, 2} (x).
		\end{split}
		\end{equation}
		Here, $\tilde u_{N_1}$ and $\tilde u_{N_2}$ satisfy (\ref{cn1}) with corresponding $f_1$ and $f_2$.
		Then, we have
			\begin{align*}
			&\|\tilde u_{N_1}-\tilde u_{N_2}\|_{L^4([0, T]; L^2(\Omega))}\leq\\ &O\left(\varepsilon^{1/2}+\frac{\varepsilon}{{\lambda}^{1/4}}+\|u_{0, 1}-u_{0, 2}\|_{L^2(\Omega)}+\|f_1 -f_2\|_{L^4([0, T]; L^2(\Omega))}\right).
				\end{align*}
		%$$\sup_{[0, T]}\|\tilde u_{N_1}-\tilde u_{N_2}\|^2_{L^2(\Omega)}\leq O({\varepsilon}).$$
	\end{theorem}

	\section{Elliptic Equations: Proofs of Main Theorems}\label{Sec:elliptic}
	
	%\subsection{Background on Elliptic Case}
	%The simplest nontrivial examples of elliptic PDEs are the Laplace equation $\Delta u = 0,$ and the Poisson equation $\Delta u = f(x,y)$. Elliptic equations  are well suited to describe equilibrium states, where any discontinuities are already been smoothed out.
	%Considering
	%\begin{equation} \label{a1n}
	%\left\{
	%\begin{split}
	%\mathscr{L}u_N &= f_N, \\
	%u_N|_{\partial \Omega } &= g_N,
	%\end{split}
	%\right.
	%\end{equation}

	\textbf{Proof of Theorem \ref{theorem1}:}
	
	\begin{proof}
	We remark first that given any $\epsilon>0$, by Theorem \ref{theorappr}, there exists a DNN ${\mathscr F}_N$ of complexity $N$ and $u_N \in {\mathscr F}_N$ such that $\|u-u_N\|_{H^2(\Omega)} \le \epsilon$.
		\begin{align*}
			\|\mathscr{L}u_N-f\|_{L^2(\Omega)}^2&= \|\mathscr{L}u_N-\mathscr{L}u\|_{L^2(\Omega)}^2 \\ \nonumber
			&\leq C_{\mathscr{L}}^2\|u_N-u\|_{H^2(\Omega)}^2 \\ \nonumber
			&\leq C_{\mathscr{L}}^2\varepsilon^2,
		\end{align*}
		where $C_{\mathscr{L}}$ is the operator norm bound of 
		${\mathscr{L}}$.
		Therefore
		\begin{align}\label{j3}
			\alpha^2\|\mathscr{L}u_N-f\|_{L^2(\Omega)}^2\leq C_{\mathscr{L}}\alpha^2 \varepsilon^2.
		\end{align}
		We also  have
		\begin{align*}
			\|u_N|_{\partial \Omega}\|_{L^2(\partial \Omega)}^2&= \|u_N|_{\partial \Omega}-u|_{\partial \Omega}\|_{L^2(\partial \Omega)}^2 \\ \nonumber
			&\leq C_{Tr}\|u_N-u\|_{H^2(\Omega)}^2 \\ \nonumber
			&\leq C_{Tr}\varepsilon^2,
		\end{align*}
		where $C_{Tr}$ is the constant from the trace operator.
		Therefore
		\begin{align}\label{j4}
			\beta^2\| u_N|_{\partial \Omega}\|_{L^{2}(\partial \Omega)}^2\leq C_{Tr}\beta^2 \varepsilon^2.
		\end{align}
		Combining (\ref{j3}) and (\ref{j4}), we have 
		\[
		\displaystyle \alpha^2 \|\mathscr{L}u_N-f\|_{L^2(\Omega)}^2 + \beta^2\| u_N|_{\partial \Omega}\|_{L^{2}(\partial \Omega)}^2\leq c\varepsilon^2.
		\]
		Finally,
		\begin{align*}
			\| u_N\|_{H^{2}(\Omega)}&\leq \|u_N-u\|_{H^{2}(\Omega)}+\|u\|_{H^{2}(\Omega)}\\ \nonumber
			&\leq {\varepsilon}+M=\widetilde{M}.
		\end{align*}	
		%Choosing $\lambda$ such that
		%\begin{align}\label{j5}
		%\lambda \| u_N\|_{H^{2}(\Omega)}^2<\lambda({\varepsilon}^2+M)<\frac{\varepsilon^2}{3}.
		%\end{align}
		%Combining (\ref{j3})-(\ref{j5}), we have
		%\begin{align*}
		%\alpha^2 \|\mathscr{L}u_N-f\|_{L^2(\Omega)}^2 + \beta^2\| u_N|_{\partial \Omega}-g\|_{L^{2}(\partial \Omega)}^2+\lambda \| u_N\|_{H^{2}(\Omega)}^2 < \varepsilon^2.
		%\end{align*}
		
		%\begin{equation}\label{a4}
		%\begin{split}
		%&\|\mathscr{L}u_N-f\|_{L^2(\Omega)}^2 + \|u_N|_{\partial \Omega}-g\|_{L^2(\partial \Omega)}^2 \\
		%&\leq \|\mathscr{L}u_N-\mathscr{L}u\|_{L^2(\Omega)}^2 +\|u_N|_{\partial \Omega}-u|_{\partial \Omega}\|_{L^2(\partial \Omega)}^2 \\
		%&\leq C_{\mathscr{L}}\|u_N-u\|_{H^2(\Omega)}^2 + C_{Tr}\|u_N-u\|_{H^2(\Omega)}^2 \\
		%&= \left( C_{\mathscr{L}} + C_{Tr} \right)\varepsilon^2.
		%\end{split}
		%\end{equation}
		
		%This means the infimum in \eqref{a3} can be made small.
	\end{proof}
	%Let us consider a specific example of $\mathscr{L}$ being a uniformly elliptic operator (a particular example is the Laplace operator) and $g=0$, i.e., we consider the homogeneous Dirichlet boundary condition.
	
	Let us consider the converse of Theorem \ref{theorem1}. Same as in the previous settings, $u$ is the unique solution of \eqref{a1} and $\mathscr{F}_N$ is a DNN. We have the following results.
	%\end{equation}
	%for some $\alpha, \beta, \lambda> 0.$
	%Moreover, we assume that the boundary condition can be encoded in $\mathscr{F}_N$ (i.e. $u|_{\partial \Omega} = 0$ for all $u \in \mathscr{F}_N$).
	
	%Let $u_N^{opt}$ solve the minimization problem with the infimum in \eqref{a3} is smaller than $\varepsilon$. Denote $\mathscr{L}u_N^{opt} = \tilde{f}$, then, $||f-\tilde{f}||_{L^2(\Omega)} < \varepsilon$.
	%\begin{lemma}
	%	Let $u$ be a solution of \eqref{a1} and $u_N\in \mathscr{F}_N$ such that $\displaystyle \alpha^2 \|\mathscr{L}u_N-f\|_{L^2(\Omega)}^2 + \beta^2\| u_N|_{\partial \Omega}\|_{L^{2}(\partial \Omega)}^2+\lambda \| u_N\|_{H^{2}(\Omega)}^2 < \varepsilon^2$, then
	%	$$\|u-u_N\|_{L^2(\Omega)} \leq \varepsilon.$$
	%\end{lemma}
	%We will now consider the converse, i.e., . What can we say about $||u^{sol}-u_N^{opt}||$ in some norm?
\begin{lemma}\label{aprlem}
 If $\displaystyle \| u_N\|_{H^{2}(\Omega)}\leq \widetilde{M},$ $\displaystyle \| \nabla (u-u_N)\|_{L^{2}(\Omega)}^2\leq c \varepsilon (M+\widetilde{M}) \left(\frac{1}{\alpha}+\frac{1}{\beta}\right),$ and $\displaystyle \| (u-u_N)|_{\partial \Omega}\|_{L^{2}(\partial \Omega)}^2<\frac{\varepsilon^2}{\beta^2}$, then 
	\begin{align} \label{l2u}
\|u-u_N\|_{L^{2}(\Omega)}^2< O((M+\widetilde{M}) \varepsilon).
	\end{align}
	\end{lemma}
\begin{proof}
			Since
		\begin{align*}
	\|u-u_N\|_{L^2(\Omega)}&\leq \|u\|_{L^2(\Omega)}+\|u_N\|_{L^2(\Omega)}\\ \nonumber
	&\leq \|u\|_{L^2(\Omega)}+\|u_N\|_{H^2(\Omega)},
	\end{align*}	
	and
	\begin{align*}
	\|u\|_{L^2(\Omega)}\leq \|u\|_{H^2(\Omega)}\leq M;\ \|u_N\|_{H^2(\Omega)}\leq \widetilde{M},
	\end{align*}
	we have
	\begin{align}\label{apr1}
	\|u-u_N\|_{L^2(\Omega)}\leq M+\widetilde{M}\ \text{and}\ 	\|u-u_N\|_{H^2(\Omega)}\leq M+\widetilde{M}.
	\end{align}
	Consider
	\begin{align}\label{j7}
	\| (u-u_N)|_{\partial \Omega}\|_{H^{1/2}({\partial \Omega})}&\leq \| (u-u_N)|_{\partial \Omega}\|_{L^{2}({\partial \Omega})}^{2/3} \| (u-u_N)|_{\partial \Omega}\|_{H^{3/2}({\partial \Omega})}^{1/3}\\ \nonumber
	&\leq c \frac{\varepsilon^{2/3}}{\beta^{2/3}}  \| (u-u_N)\|_{H^{2}({\Omega})}^{1/3}\\\nonumber
	&\leq c(M+\widetilde{M})^{1/3}  \frac{\varepsilon^{2/3}}{\beta^{2/3}}.
	\end{align}
	
	Consider the lifting operator $\mathnormal{l}_{\Omega}: H^{1/2}(\partial \Omega) \to H^1(\Omega),$ which is linear and bounded such that $Tr \mathnormal{l}_{\Omega} = I.$ Here, $Tr$ is the trace operator and $I$ is the identity operator.
	\begin{align*}
	\| \mathnormal{l}_{\Omega} (u-u_N)\|_{H^{1}({\Omega})}\leq C_{\mathnormal{l}_{\Omega}} \| (u-u_N)|_{\partial \Omega}\|_{H^{1/2}({\partial \Omega})}\leq c(M+\widetilde{M})^{1/3} \frac{\varepsilon^{2/3}}{\beta^{2/3}},
	\end{align*}
	where $C_{\mathnormal{l}_{\Omega}}$ is the constant from the lifting operator.
	Moreover
	\begin{align*}
	\| (u-u_N)-\mathnormal{l}_{\Omega} (u-u_N)\|_{L^{2}({\Omega})}&\leq c\| \nabla ((u-u_N)-\mathnormal{l}_{\Omega} (u-u_N))\|_{L^{2}({\Omega})}\\
	&\leq c\| \nabla (u-u_N)\|_{L^{2}({\Omega})}+c\| \nabla (\mathnormal{l}_{\Omega} (u-u_N))\|_{L^{2}({\Omega})}\\
	&\leq c \varepsilon^{1/2} (M+\widetilde{M})^{1/2} \left(\frac{1}{\alpha}+\frac{1}{\beta}\right)^{1/2}+c\| \mathnormal{l}_{\Omega} (u-u_N)\|_{H^{1}({\Omega})}\\
	&\leq c \varepsilon^{1/2} (M+\widetilde{M})^{1/2} \left(\frac{1}{\alpha}+\frac{1}{\beta}\right)^{1/2}+c(M+\widetilde{M})^{1/3} \frac{\varepsilon^{2/3}}{\beta^{2/3}}.
	\end{align*}
	Therefore
	\begin{align*}
	\|u-u_N\|_{L^{2}({\Omega})}&\leq \|(u-u_N)-\mathnormal{l}_{\Omega} (u-u_N)+\mathnormal{l}_{\Omega} (u-u_N)\|_{L^{2}({\Omega})}\\ \nonumber
	&\leq \| (u-u_N)-\mathnormal{l}_{\Omega} (u-u_N)\|_{L^{2}({\Omega})}+\| \mathnormal{l}_{\Omega} (u-u_N)\|_{L^{2}({\Omega})}\\ \nonumber
	%&\leq c \varepsilon^{1/2}(M+\widetilde{M})^{1/2}  \left(\frac{1}{\alpha}+\frac{1}{\beta}\right)^{1/2}+c(M+\widetilde{M})^{1/3} \frac{\varepsilon^{2/3}}{\beta^{2/3}}+c(M+\widetilde{M})^{1/3} \frac{\varepsilon^{2/3}}{\beta^{2/3}}\\ \nonumber
	&\leq c \varepsilon^{1/2}(M+\widetilde{M})^{1/2}  \left(\frac{1}{\alpha}+\frac{1}{\beta}\right)^{1/2}+c(M+\widetilde{M})^{1/3} \frac{\varepsilon^{2/3}}{\beta^{2/3}}\\ \nonumber
	&=O((M+\widetilde{M})^{1/2} \varepsilon^{1/2}).
	\end{align*}	
	\end{proof}
	%We will now consider the converse, i.e., . What can we say about $||u^{sol}-u_N^{opt}||$ in some norm?
	\textbf{Proof of Theorem \ref{thm2}:}
	\begin{proof}
		First, we have
		\begin{equation}\label{j1}
			\| (u-u_N)|_{\partial \Omega}\|_{L^{2}(\partial \Omega)}^2=\| u_N|_{\partial \Omega}\|_{L^{2}(\partial \Omega)}^2\leq\frac{\varepsilon^2}{\beta^2}.
		\end{equation}
		
		Moreover
		\begin{align}\label{j2}
			&\| \nabla (u-u_N)\|_{L^{2}(\Omega)}^2\\ \nonumber
			&=\int_{\Omega} (\nabla (u-u_N))^2 dx\\ \nonumber
			&=\int_{\partial \Omega}\nabla (u-u_N)|_{\partial \Omega} \cdot (u-u_N)|_{\partial \Omega}dx-\int_{\Omega}  (u-u_N)\cdot  (\Delta (u-u_N))dx\\ \nonumber
			&\leq \|\nabla (u-u_N)|_{\partial \Omega}\|_{L^2(\partial \Omega)}\cdot  \|(u-u_N)|_{\partial \Omega}\|_{L^2(\partial \Omega)}+\|u-u_N\|_{L^2(\Omega)}\cdot  \|\Delta (u-u_N)\|_{L^2(\Omega)}.
		\end{align}	
		
		Since
		\begin{align*}
			\|\Delta (u-u_N)\|_{L^2(\Omega)}\leq c\|\mathscr{L}u_N-f\|_{L^2(\Omega)}\leq c\frac{\varepsilon}{\alpha},
		\end{align*}	
		\begin{align*}
			\|u-u_N\|_{L^2(\Omega)}&\leq \|u\|_{L^2(\Omega)}+\|u_N\|_{L^2(\Omega)}\\ \nonumber
			&\leq \|u\|_{L^2(\Omega)}+\|u_N\|_{H^2(\Omega)},
		\end{align*}	
		and
		\begin{align*}
			\|u\|_{L^2(\Omega)}\leq \|u\|_{H^2(\Omega)}\leq M;\ \|u_N\|_{H^2(\Omega)}\leq \widetilde{M},
		\end{align*}
		we have
		\begin{align*}
			\|u-u_N\|_{L^2(\Omega)}\leq M+\widetilde{M}.
		\end{align*}
		Moreover
		\begin{align*}
			\|(u-u_N)|_{\partial \Omega}\|_{L^2(\partial \Omega)}\leq \frac{\varepsilon}{\beta},
		\end{align*}
		and
		\begin{align*}
			\|\nabla (u-u_N)|_{\partial \Omega}\|_{L^2(\partial \Omega)}&\leq c\|u-u_N\|_{H^2(\Omega)}\\ \nonumber
			&\leq c(\|u\|_{H^2(\Omega)}+\|u_N\|_{H^2(\Omega)}) <c(M+\widetilde{M}).
		\end{align*}
		Therefore, we have
		\begin{align}\label{grau}
			\| \nabla (u-u_N)\|_{L^{2}(\Omega)}^2\leq c \varepsilon (M+\widetilde{M})\left(\frac{1}{\alpha}+\frac{1}{\beta}\right).
		\end{align}
		Therefore, we can apply Lemma \ref{aprlem} and obtain (\ref{l2u}).
		Combining (\ref{grau}) and (\ref{l2u}), we have
		\begin{align*}
			\|u-u_N\|_{H^1(\Omega)}= \|u-u_N\|_{L^{2}({\Omega})}+\| \nabla (u-u_N)\|_{L^{2}(\Omega)}=O((M+\widetilde{M})^{1/2} \varepsilon^{1/2}).
		\end{align*}
		% This implies
		%
		%\begin{equation}\label{a5}
		%||u^{sol}-u_N^{opt}||_{H^2(\Omega)} = ||\mathscr{L}^{-1}(f-\tilde{f})||_{H^2(\Omega)} \leq C_{\mathscr{L}^{-1}}||f-\tilde{f}||_{L^2(\Omega)}<C_{\mathscr{L}^{-1}}\varepsilon,
		%\end{equation}
		%if $\mathscr{L}^{-1} : L^2(\Omega) \to H^2(\Omega)$ is a bounded linear operator.
		% the following is from page 12 of the hand written notes.
		%Our goal is to show that $||u^{sol}-u_N^{opt}||< \delta$.
		
		We can improve the rate of $\|u-u_N\|_{L^{2}({\Omega})}$ by using a different approach:\\
		%We note that even if $u$ were a weak solution of \eqref{a1} with $u\in H^1(\Omega),$ then $u|_{\partial \Omega} \in H^{1/2}(\partial \Omega)$. There exists a lifting operator $\mathnormal{l}_{\Omega}: H^{1/2}(\partial \Omega) \to H^1(\Omega),$ which is linear and bounded such that $Tr \mathnormal{l}_{\Omega} = I.$
		Denote $\mathscr{L}u_N = f_{\varepsilon}$ and $u_N|_{\partial \Omega} = Tr(u_N) = g_{\varepsilon}$. From \eqref{j8}, we have $\displaystyle \|f-f_{\varepsilon}\|_{L^2(\Omega)}\leq \frac{\varepsilon}{\alpha}$, from \eqref{j7}, we have $\displaystyle \|g_{\varepsilon}\|_{H^{1/2}(\partial \Omega)}< c(M+\widetilde{M})^{1/3}\frac{\varepsilon^{2/3}}{\beta^{2/3}}$. Let $\tilde{u}_N = u_N-\mathnormal{l}_{\Omega}g_{\varepsilon}$. Then
		
		%\begin{equation}\label{a7}
		%\begin{split}
		%\mathscr{L}\tilde{u}^{sol} = f - \mathscr{L}\mathnormal{l}_{\Omega}g,\\
		%\tilde{u}^{sol}|_{\partial \Omega} = 0.
		%\end{split}
		%\end{equation}
		
		\begin{equation}\label{a8}
			\begin{split}
				\mathscr{L}\tilde{u}_{N} = f_{\varepsilon} - \mathscr{L}\mathnormal{l}_{\Omega}g_{\varepsilon},\\
				\tilde{u}_{N}|_{\partial \Omega} = 0.
			\end{split}
		\end{equation}
		Note that since $\mathscr{L}$ is a second order elliptic operator and $\mathnormal{l}_{\Omega}g_{\varepsilon} \in H^{1}(\Omega)$, we have $\mathscr{L}\mathnormal{l}_{\Omega}g_{\varepsilon} \in H^{-1}(\Omega)$. From Lax-Milgram \cite{Evans}, we have
		
		\begin{equation}\label{a9}
			\|u-\tilde{u}_N\|_{H^1(\Omega)} \leq  C_{\mathscr{L}^{-1}}\| (f-f_{\varepsilon})+\mathscr{L}\mathnormal{l}_{\Omega}g_{\varepsilon} \|_{H^{-1}(\Omega)}.
		\end{equation}
		Therefore, we have
		\begin{equation}\label{a10}
			\begin{split}
				&\|u-u_N\|_{H^1(\Omega)}\\
				&= \| u -(u_N-\mathnormal{l}_{\Omega}g_{\varepsilon}) -\mathnormal{l}_{\Omega}g_{\varepsilon} \|_{H^1(\Omega)}\\
				&\leq \|u-\tilde{u}_N\|_{H^1(\Omega)} + \|\mathnormal{l}_{\Omega}g_{\varepsilon} \|_{H^1(\Omega)}\\
				&\leq C_{\mathscr{L}^{-1}}\| (f-f_{\varepsilon}) + \mathscr{L}\mathnormal{l}_{\Omega}g_{\varepsilon}) \|_{H^{-1}(\Omega)} + \|\mathnormal{l}_{\Omega}\|\|g_{\varepsilon}\|_{H^{1/2}(\partial \Omega)}.
			\end{split}
		\end{equation}
		Since
		$$\|f-f_{\varepsilon}\|_{H^{-1}(\Omega)} \leq c \|f-f_{\varepsilon}\|_{L^2(\Omega)}$$
		and
		$$\|\mathscr{L}\mathnormal{l}_{\Omega}g_{\varepsilon}\|_{H^{-1}(\Omega)} \leq C_{\mathscr{L}}\|\mathnormal{l}_{\Omega}g_{\varepsilon} \|_{H^1(\Omega)}\leq C_{\mathscr{L}}\|\mathnormal{l}_{\Omega}\|\|g_{\varepsilon}\|_{H^{1/2}(\partial \Omega)}.$$
		Thus
		\begin{equation}\label{a11}
			\begin{split}
				\|u-u_N\|_{L^2(\Omega)} & \leq \| u - u_N\|_{H^1(\Omega)} \\
				& \leq c\|f-f_{\varepsilon}\|_{L^2(\Omega)} + c \|g_{\varepsilon}\|_{H^{1/2}(\partial \Omega)}.
			\end{split}
		\end{equation}
		Therefore, we have
		\begin{equation}\label{a12}
			\|u-u_N\|_{L^2(\Omega)} \leq \frac{c}{\alpha} \varepsilon + \frac{c(M+\widetilde{M})^{1/3}}{\beta^{2/3}}\varepsilon^{2/3} =: O((M+\widetilde{M})^{1/3}\varepsilon^{2/3}).
		\end{equation}
	\end{proof}
	We will show below that by considering the loss function to be
	$$\alpha^2 \|\mathscr{L}u_N-f\|_{L^2(\Omega)}^2 + \beta^2\| u_N|_{\partial \Omega}\|_{H^{1/2}(\partial \Omega)}^2,$$
	we have an improved result on $\|u-u_N\|_{L^2(\Omega)}$ which is stated in Theorem \ref{thmay1}.\\[5pt]

\textbf{Proof of Theorem \ref{thmay1}:}
	\begin{proof}
		Same as before, $\mathscr{L}u_N = f_{\varepsilon}$ and $u_N|_{\partial \Omega} = Tr(u_N) = g_{\varepsilon}$. From \eqref{j12}, we have $\displaystyle \|f-f_{\varepsilon}\|_{L^2(\Omega)}\leq \frac{\varepsilon}{\alpha}$,
		and $\displaystyle \|g_{\varepsilon}\|_{H^{1/2}(\partial \Omega)}\leq\frac{\varepsilon}{\beta}$.
		Let $\tilde{u}_N = u_N-\mathnormal{l}_{\Omega}g_{\varepsilon}$.
		%Then
		%
		%\begin{align*}
		%\| g-g_{\varepsilon}\|_{H^{1/2}({\partial \Omega})}=\| (u-u_N)|_{\partial \Omega}\|_{H^{1/2}({\partial \Omega})}&\leq \| (u-u_N)|_{\partial \Omega}\|_{L^{2}({\partial \Omega})}^{2/3} \| (u-u_N)|_{\partial \Omega}\|_{H^{3/2}({\partial \Omega})}^{1/3}\\ \nonumber
		%&\leq \| (u-u_N)|_{\partial \Omega}\|_{H^{1/2}({\partial \Omega})}^{2/3} \| (u-u_N)|_{\partial \Omega}\|_{H^{3/2}({\partial \Omega})}^{1/3}\\\nonumber
		%&\leq c \frac{\varepsilon^{2/3}}{\beta^{2/3}}  \| (u-u_N)\|_{H^{2}({\Omega})}^{1/3}\\\nonumber
		%&\leq c(M+\widetilde{M})^{1/3}  \frac{\varepsilon^{2/3}}{\beta^{2/3}}:=c \frac{\varepsilon^{2/3}}{\beta^{2/3}}.
		%\end{align*}

		%\begin{equation}\label{a7}
		%\begin{split}
		%\mathscr{L}\tilde{u}^{sol} = f - \mathscr{L}\mathnormal{l}_{\Omega}g,\\
		%\tilde{u}^{sol}|_{\partial \Omega} = 0.
		%\end{split}
		%\end{equation}
		
		\begin{equation}\label{a8n}
			\begin{split}
				\mathscr{L}\tilde{u}_{N} = f_{\varepsilon} - \mathscr{L}\mathnormal{l}_{\Omega}g_{\varepsilon},\\
				\tilde{u}_{N}|_{\partial \Omega} = 0.
			\end{split}
		\end{equation}
		%	Note that since $\mathscr{L}$ is a second order elliptic operator and $\mathnormal{l}_{\Omega}g \in H^{1}(\Omega)$, we have $\mathscr{L}\mathnormal{l}_{\Omega}g \in H^{-1}(\Omega)$. Applying the usual theory (Lax-Milgram, c.f. Evan's PDE book \cite{Evans}),
		%	
		Similar to the proof of Theorem \ref{thm2}, we have
		\begin{equation*}
			\begin{split}
				\|u-u_N\|_{H^1(\Omega)}
				\leq C_{\mathscr{L}^{-1}}\| (f-f_{\varepsilon}) + \mathscr{L}\mathnormal{l}_{\Omega}g_{\varepsilon}) \|_{H^{-1}(\Omega)} + \|\mathnormal{l}_{\Omega}\|\|g_{\varepsilon}\|_{H^{1/2}(\partial \Omega)},
			\end{split}
		\end{equation*}
		and
		\begin{equation*}
			\begin{split}
				\|u-u_N\|_{L^2(\Omega)} & \leq \| u - u_N\|_{H^1(\Omega)} \\
				& \leq c\|f-f_{\varepsilon}\|_{L^2(\Omega)} + c \|g_{\varepsilon}\|_{H^{1/2}(\partial \Omega)}.
			\end{split}
		\end{equation*}
		Therefore, we have
		\begin{equation*}
			\|u-u_N\|_{L^2(\Omega)} \leq \frac{c}{\alpha} \varepsilon + \frac{c}{\beta}\varepsilon = O(\varepsilon).
		\end{equation*}
	\end{proof}

	\section{Navier-Stokes Equations: Proof of Main Theorems}\label{section:navier-stokes}
	%\subsection{Background on Navier-Stokes Equations}
	
	\subsection{Functional analytic framework.}
	Let $\Omega$ be a bounded domain in $\R^2$ and
	$$H := \left \{u\in L^2(\Omega),\ \nabla \cdot u \in L^2(\Omega),\ \nabla \cdot u=0\right \},$$
	$$V := \left \{u\in H^1_0(\Omega),\ \nabla \cdot u=0\right \}.$$
	%$H$  be divergence free $L^2$ and $V$ be divergence free $H_0^1(\Omega)$.
	$V'$ is the dual space of $V$.\\
	Let $\mathbb{P}$ be the Leray Projection which is an orthogonal projection from $L^2$ onto the subset of $L^2$ consisting of those functions whose weak derivatives are divergence-free in the $L^2$ sense. $A$ is the Stokes operator, defined as
	$\displaystyle
	\label{defA}
	A=-\mathbb{P} \Delta.
	$ $B$ is the bilinear form defined by
	$\displaystyle
	\label{defB}
	B(u,u)=\mathbb{P} \left [({u} \cdot \nabla) {u}\right ].
	$
	
	Applying the projection $\mathbb{P}$ on (\ref{c2}),  the functional form of the NSE can be written as
	\begin{align}
		\label{functional_form}
		&\frac{du}{dt}+\mathit{A}u+\mathit{B}(u,u)=\mathbb{P}f,\\ \nonumber
		&u|_{\partial \Omega}= 0.
		%&u(x, 0)=u_0(x).
	\end{align}
	
	We recall the definition of strong solutions from \cite{temam1995navier}:\\
	Let $\displaystyle W=\left \{u\in H^1_{loc}(\Omega)\ \text{and}\ \nabla \cdot {u}=0\ \text{in}\ \Omega\right \}$ and $u_0\in W$, $u$ is a strong solution of NSE if it solves the variational formulation of (\ref{c2}) as in \cite{CF, temam1995navier}, and
	$$u\in L^2(0, T; D(A))\cap L^{\infty}(0, T; W),$$
	for $T>0$.
	
	\subsection{Hodge decomposition.}
	The idea of Hodge decomposition is to decompose a vector 
	$u \in L^2(\Omega)$ uniquely into a divergence-free part $u_1$ and an irrotational part $u_2$, which is orthogonal in $L^2(\Omega)$ to $u_1$:
	$$u=u_1+u_2,\ \nabla \cdot u_1=0\ \mbox{and}\ (u_1,u_2)=0.$$
	When we apply the Leray Projection $\mathbb{P}$ on $u$, we have
	$$\mathbb{P} u=u_1.$$
	More precisely, we have the following proposition, the proof of which can be found in   \cite{CF}.
	\begin{proposition}
	Let $\Omega$ be open, bounded, connected with boundary of class $C^2$. Then $L^2(\Omega)= H\oplus H_1 \oplus H_2$, where $H, H_1, H_2$ are mutually orthogonal spaces and moreover
	\[
	H_1= \{u \in L^2(\Omega)| u = \nabla p, p \in H^1(\Omega), \Delta p=0\},
	\]
	and
	\[
	H_2=\{u \in L^2(\Omega)| u = \nabla p, p \in H_0^1(\Omega)\}.
	\]
	\end{proposition}
	The decomposition above is obtained as follows. Let $v \in L^2(\Omega)$. Then, 
	\[
	v=u + u_1+u_2,\ u \in H,\ \mbox{and}\ 
	u_2= \nabla p_2, \ 
	\Delta p_2=\nabla \cdot v \in H^{-1}(\Omega),\ p_2 \in H_0^1(\Omega).
	\]
	Subsequently, $u_1$ is obtained by solving the Neumann problem
	\[
	u_1= \nabla p_1,\ \Delta p_1=0, \ 
	\frac{\partial p_1}{\partial n_\Omega}=\gamma(v-u_2),
	\]
	where $n_\Omega$ is the unit normal vector on the boundary of $\Omega$ and $\gamma$ denotes the normal trace on the boundary (see \cite{CF, Temam} for more details).
	
	%\subsection{Mathematical analysis of deep learning techniques in Navier-Stokes Equations}
	\subsection{Proofs.}
	Consider an approximate solution 
	$\tilde{u}_N \in \mathscr{F}_N$, i.e. $\tilde{u}_N$ satisfies 
	\eqref{cn1} and denote $\tilde{u}_N |_{\partial \Omega} =\tilde{g},\ \nabla \cdot \tilde{u}_N = \tilde{h}.$ Let
	$\displaystyle
	\tilde{f} := \partial_t \tilde{u}_N - \Delta \tilde{u}_N + \tilde{u}_N  \cdot \nabla \tilde{u}_N + \nabla \tilde{p}_N - f.
	$
	Then
	\begin{equation}\label{cn4}
		\begin{split}
			\partial_t \tilde{u}_N - \Delta \tilde{u}_N + \tilde{u}_N \cdot \nabla \tilde{u}_N + \nabla \tilde{p}_N &= f + \tilde{f}, \\
			\nabla \cdot \tilde{u}_N &= \tilde{h}, \\
			\tilde{u}_N|_{\partial \Omega} &= \tilde{g}.
		\end{split}
	\end{equation}
	Applying the Hodge decomposition on $\tilde{u}_N$:
	\begin{equation}\label{c5nn}
		%\label{c5}
		\tilde{u}_N = \mathbb{P} \tilde{u}_N + (\mathbb{I} - \mathbb{P})\tilde{u}_N =: u_N + v_N,
	\end{equation}
	where $u_N=\mathbb{P} \tilde{u}_N$, $\nabla \cdot u_N=0,$ $u_N|_{\partial \Omega}=0,$ and $v_N=(\mathbb{I} - \mathbb{P})\tilde{u}_N$.
	
	Before we prove our main theorems, we first introduce two Lemmas.
	
	\begin{lemma}\label{lemma1}
		Consider $u_N$ satisfying
		\begin{align}\label{c15}
			&\frac{du_N}{dt}  + A u_N + B(u_N, u_N) = \mathbb{P}f + \varphi,\\ \nonumber
			& u_N|_{\partial \Omega} =0.
		\end{align}
		where
		\begin{equation}\label{c16}
			\int_0^T \|\varphi\|^2_{V'} \, dt \leq O \left(\varepsilon+\frac{\varepsilon^2}{\sqrt{\lambda}}\right),
		\end{equation}
		and let  $u$ be  a strong solution of (\ref{functional_form}), with $\displaystyle \|{u}_N (x, 0) - u_0(x)\|_{L^2}^2\leq {\varepsilon}^2.$
%		\begin{equation}\label{c17}
%			\begin{split}
%				&\partial_t u + Au + B(u,u)= Pf,\\
%				&u|_{\partial \Omega}= 0.
%			\end{split}
%		\end{equation}
		
		%	Assume
		%	$$\|(g_1-g_2)|_{\partial \Omega}\|_{L^4([0, T]; H^{1/2}(\partial \Omega))}^4+\int_0^T \|\varphi_{\varepsilon}\|^2_{V'} \, dt \leq c\varepsilon^2.$$
		%	$$\int_0^T \|\varphi\|^2_{V'} \, dt \leq c\varepsilon^2,$$
		%	then \begin{align*}
		%	\sup_{[0, T]}\|u(x, t)-u_N(x, t)\|_{L^2(\Omega)}^2\leq O({\varepsilon}^2).
		%	\end{align*}
		%	$$\int_0^T \|\varphi\|^2_{V'} \, dt \leq O \left(\varepsilon+\frac{\varepsilon^2}{\sqrt{\lambda}}\right),$$
		Then \begin{align*}
			\sup_{[0, T]}\|u(x, t)-u_N(x, t)\|_{L^2(\Omega)}^2\leq O\left(\varepsilon+\frac{\varepsilon^2}{\sqrt{\lambda}}\right).
		\end{align*}
	\end{lemma}
	\begin{proof}
		Considering $w(t)=u(t)-u_N(t)$, from (\ref{functional_form}) and (\ref{c15}), we have
		\begin{equation}\label{c18}
			\begin{split}
				\frac{dw}{dt} + Aw + B(u,u)- B(u_N, u_N) &=-\varphi,\\
				w|_{\partial \Omega}= 0.
			\end{split}
		\end{equation}
		
		Since
		\begin{align*}
			B(u,u)- B(u_N, u_N)=B(u,w)+ B(w, u_N)=B(u,w)+B(w, u)-B(w, w),
		\end{align*}
		
		we have
		\begin{equation}\label{c19}
			\frac{dw}{dt} + Aw + B(u,w)+B(w, u)-B(w, w)=-\varphi.
		\end{equation}
		
		Taking inner product of (\ref{c19}) with $w$, we obtain
		\begin{align*}
			\frac{1}{2}\frac{d}{dt}\|w\|_{L^2}^2+\|A^{1/2}w\|_{L^2}^2+(B(w,u),w)=-(\varphi,w).
		\end{align*}
		
		Therefore
		\begin{align*}
			\frac{1}{2}\frac{d}{dt}\|w\|_{L^2}^2+\|A^{1/2}w\|_{L^2}^2\leq |(B(w,u),w)|+|(\varphi,w)|.
		\end{align*}
		
		Since
		\begin{align*}
			\left |(B(w,u),w)\right |\leq c\|A^{1/2}u\|_{L^2} \|w\|_{L^2} \|A^{1/2}w\|_{L^2}\leq c\frac{\|A^{1/2}u\|_{L^2}^2 \|w\|_{L^2}^2}{2}+\frac{\|A^{1/2}w\|_{L^2}^2}{2},
		\end{align*}
		and
		\begin{align*}
			|(\varphi,w)|\leq  \|\varphi\|_{V'}\|A^{1/2}w\|_{L^2}\leq \frac{\|\varphi\|_{V'}^2}{2}+\frac{\|A^{1/2}w\|_{L^2}^2}{2},
		\end{align*}
		we have
		\begin{align*}
			\frac{1}{2}\frac{d}{dt}\|w\|_{L^2}^2-c\frac{\|A^{1/2}u\|_{L^2}^2 \|w\|_{L^2}^2}{2}\leq \frac{\|\varphi\|_{V'}^2}{2}.
		\end{align*}
		
		Equivalently
		\begin{align*}
			\frac{d}{dt}\|w\|_{L^2}^2-c\|A^{1/2}u\|_{L^2}^2 \|w\|_{L^2}^2\leq \|\varphi\|_{V'}^2.
		\end{align*}
		
		Applying Gronwall's inequality, we have
		\begin{align*}
			\|w(t)\|_{L^2}^2\leq e^{\int^t_0 c\|A^{1/2}u\|_{L^2}^2 ds}\|w(0)\|_{L^2}^2+e^{\int^t_0 c\|A^{1/2}u\|_{L^2}^2 ds}\int^t_0 e^{-\int^s_0 c\|A^{1/2}u\|_{L^2}^2 d\tau}\|\varphi\|_{V'}^2ds.
		\end{align*}
		
		Since
		$$\int^t_0 c\|A^{1/2}u\|_{L^2}^2 ds\leq cF(G)t,$$
		where, $G$ is the Grashof number (defined as $G=\|f\|$ \cite{foias2002statistical}) and $F(G)$ is a function of $G$, we have
		\begin{align*}
			e^{\int^t_0 c\|A^{1/2}u\|_{L^2}^2 ds}\leq e^{cF(G)t}.
		\end{align*}
		
		Moreover
		\begin{align*}
			e^{-\int^s_0 c\|A^{1/2}u\|_{L^2}^2 d\tau}\leq 1,
		\end{align*}
		and
		\begin{align*}
			\|w(0)\|_{L^2}^2= \|{u}_N (x, 0) - u_0(x)\|_{L^2}^2\leq {\varepsilon}^2.
		\end{align*}
		
		We have
		\begin{align*}
			\|w(t)\|_{L^2}^2&\leq e^{cF(G)t}\|w(0)\|_{L^2}^2+e^{cF(G)t}\int^t_0 ||\varphi||_{V'}^2ds\\
			&\leq O\left(\varepsilon+\frac{\varepsilon^2}{\sqrt{\lambda}}\right) e^{cF(G)t}.
		\end{align*}
		
		Therefore
		\begin{align*}
			\sup_{[0, T]}\|w(t)\|_{L^2}^2\leq O\left(\varepsilon+\frac{\varepsilon^2}{\sqrt{\lambda}}\right).
		\end{align*}
	\end{proof}
	\begin{lemma}\label{lemma5.2}
		Assume
		\begin{align}\label{c2nn2}
			%\begin{split}
			\|\tilde{u}_N|_{\partial \Omega}\|_{L^4([0, T]; H^{1/2}(\partial \Omega))}^4+
			\|\nabla \cdot \tilde{u}_N\|_{L^4([0, T]; L^2(\Omega))}^4
			\leq \varepsilon^2,
			%	\end{split}
		\end{align}
		and with the Hodge decomposition (\ref{c5nn}), we have
		\begin{align}\label{c5nn1}
			\|v_N\|_{L^4([0, T]; H^1(\Omega))}\leq c{\sqrt{\varepsilon}}.
		\end{align}
	\end{lemma}
	\begin{proof}
		Consider the Hodge decomposition
		\begin{equation*}
			\tilde{u}_N = u_N + v_N=u_N\oplus v_1 \oplus v_2
		\end{equation*}
		with
		\begin{equation*}
			v_N=v_1 \oplus v_2,
		\end{equation*}
		where $v_1=\nabla p_1$ and $v_2=\nabla p_2$, $p_1$ and $p_2$ are the solutions of the following two systems, respectively.
		\begin{equation}\label{2estp1}
			\left\{
			\begin{split}
				\Delta p_1&=0,\\
				\frac{\partial p_1}{\partial n}&=\gamma(\tilde{u}_N-v_2),
			\end{split}\right.
		\end{equation}
		and
		\begin{equation}\label{2estp2}
			\left\{
			\begin{split}
				\Delta p_2&=\nabla \cdot \tilde{u}_N,\\
				Tr(p_2)&=0.
			\end{split}\right.
		\end{equation}
		Here, $\gamma$ is the trace operator and $Tr(p_2)$ means the value of $p_2$ on the boundary.
		According to Lions \& Magenes \cite{lions2012non}, the above two systems have unique solutions (up to an additive constant). First, we solve for $p_2$ from (\ref{2estp2}). Accordingly, $v_2$ can be obtained. Then, we use $v_2$ to solve for $p_1$ from (\ref{2estp1}) and find $v_1$ afterwards.
		
		From (\ref{2estp2}) and (\ref{c2nn2}), we have
		\begin{align*}
			%\label{u2}
			\|v _2\|_{L^4([0, T]; H^1(\Omega))}&\leq \| p_2\|_{L^4([0, T]; H^2(\Omega))}\\ \nonumber
			&\leq c\|\nabla \cdot \tilde{u}_N\|_{L^4([0, T]; L^2(\Omega))}\\ \nonumber
			&\leq c\varepsilon^{1/2}.
		\end{align*}
		%	\begin{align}
		%	\label{u2}
		%	\|v_2\|_{L^4([0, T]; H^1(\Omega))}\leq c\|v_2\|_{L^2([0, T]; H^1(\Omega))}\leq c\varepsilon?
		%	\end{align}
		
		From (\ref{2estp1}) and (\ref{c2nn2}),
		% \begin{align}
		%  \label{u1}
		% || u_1||_{L^4([0, T]; H^1(\Omega))}&= || p_1||_{L^4([0, T]; H^2(\Omega))}\\ \nonumber
		% &\leq  ||\gamma(\tilde{u}_N-u_2)||_{L^4([0, T]; H^{1/2}(\Omega))}\\ \nonumber
		% &=||\gamma(u_2)||_{L^4([0, T]; H^{1/2}(\Omega))}\\ \nonumber
		% &=||Tr(u_2)\cdot n_{\Omega}||_{L^4([0, T]; H^{1/2}(\Omega))}\\ \nonumber
		% &\leq c||u_2||_{L^4([0, T]; H^1(\Omega))}\\ \nonumber
		% &\leq c\frac{\sqrt{\varepsilon}}{\lambda^{1/4}}.
		% \end{align}
		%
		\begin{align}
			\label{un1}
			\|v_1\|_{L^4([0, T]; H^1(\Omega))}&\leq \| p_1\|_{L^4([0, T]; H^2(\Omega))}\\ \nonumber
			&\leq  c\|\gamma(\tilde{u}_N-v_2)\|_{L^4([0, T]; H^{1/2}(\Omega))}\\ \nonumber
			&\leq c\| \tilde{u}_N|_{\partial \Omega}\|_{L^4([0, T]; H^{1/2}(\partial \Omega))}+c\|\gamma(v_2)\|_{L^4([0, T]; H^{1/2}(\Omega))}\\ \nonumber
			&\leq c\varepsilon^{1/2}+c\|Tr(v_2)\cdot n_{\Omega}\|_{L^4([0, T]; H^{1/2}(\Omega))}\\ \nonumber
			&\leq c\varepsilon^{1/2}+c\|v_2\|_{L^4([0, T]; H^1(\Omega))}\\ \nonumber
			&\leq c{\sqrt{\varepsilon}}.
		\end{align}

		Therefore,
		\begin{align*}
			\|v_N\|_{L^4([0, T]; H^1(\Omega))}=\|v_1\|_{L^4([0, T]; H^1(\Omega))}+\|v_2\|_{L^4([0, T]; H^1(\Omega))}\leq c{\sqrt{\varepsilon}}.
		\end{align*}
	\end{proof}
Note that (\ref{c5nn1}) also implies
\begin{align*}
\left(\int_0^T \|\nabla v_N\|_{L^2(\Omega)}^4 dt \right)^{1/4}\leq c{\sqrt{\varepsilon}},\
\left( \int_0^T \|v_N\|_{L^2(\Omega)}^4 dt \right)^{1/4}\leq c{\sqrt{\varepsilon}}.
	\end{align*}
\textbf{Proof of Theorem \ref{thmay3}}
	\begin{proof}
		%Case (i): we assume $\tilde{u}_N |_{\partial \Omega} = 0.$
		%Let
		%\begin{equation}\label{c3}
		%\tilde{f}_{\varepsilon} := \partial_t \tilde{u}_N - \Delta \tilde{u}_N + \tilde{u}_N  \cdot \nabla \tilde{u}_N + \nabla \tilde{p}_N - f.
		%\end{equation}
		%Then,
		%\begin{equation}\label{c4}
		%\begin{split}
		%\partial_t \tilde{u}_N - \Delta \tilde{u}_N + \tilde{u}_N \cdot \nabla \tilde{u}_N + \nabla \tilde{p}_N &= f + \tilde{f}_{\varepsilon}, \\
		%                        \nabla \cdot \tilde{u}_N &= \tilde{h}_{\varepsilon}, \\
		%                           \tilde{u}_N|_{\partial \Omega} &= 0.
		%\end{split}
		%\end{equation}
		Applying the Leray projection operator $\mathbb{P}$ to \eqref{cn4}, one obtains, under the assumption that $\mathbb{P}f = f$,
		\begin{equation*}
			\partial_t \mathbb{P} \tilde{u}_N - \mathbb{P} \Delta \tilde{u}_N + \mathbb{P} \tilde{u}_N \cdot \nabla \tilde{u}_N = f + \mathbb{P} \tilde{f}.
		\end{equation*}
		Recall that $\mathbb{P}$ is the orthogonal projection. By Hodge decomposition
		\begin{equation*}
			%\label{c5}
			\tilde{u}_N = \mathbb{P} \tilde{u}_N + (\mathbb{I} - \mathbb{P})\tilde{u}_N =: u_N + v_N.
		\end{equation*}
		%where $u_N=\mathbb{P} \tilde{u}_N$, $\nabla \cdot u_N=0,$ $u_N|_{\partial \Omega}=0,$ and $v_N=(\mathbb{I} - \mathbb{P})\tilde{u}_N$.
		Then
		\begin{equation*}
			%\label{c6}
			\partial_t u_N + A u_N + \mathbb{P} u_N \cdot \nabla u_N + \mathbb{P}(\tilde{u}_N \cdot \nabla \tilde{u}_N -u_N \cdot \nabla u_N) = f + \mathbb{P}\tilde{f}+\mathbb{P}\Delta v_N.
		\end{equation*}
		Here, $A$ is the Stokes' operator.
		\begin{equation}
			\begin{split}
				\mathbb{P}\left(\tilde{u}_N\cdot\nabla\tilde{u}_N - u_N\cdot\nabla u_N\right) &= \mathbb{P}\left((\tilde{u}_N -u_N)\cdot \nabla \tilde{u}_N + u_N \cdot \nabla(\tilde{u}_N -u_N)\right) \\
				& = \mathbb{P}\left(v_N\cdot \nabla \tilde{u}_N + u_N \cdot \nabla v_N\right) \\
				& =: \psi.
			\end{split}
		\end{equation}
		
		Next, we will estimate $\displaystyle \int_0^T \|\psi\|_{V'}^2 \, dt.$
		
		Note that $\displaystyle \|\psi\|_{V'} = \sup_{w \in V, \|w\|_V \leq 1}  \left \langle\psi, w\right \rangle$, where
		\begin{equation}\label{c8}
			\left \langle\psi, w\right \rangle = \int_{\Omega} \mathbb{P}(v_N \cdot \nabla \tilde{u}_N + u_N \cdot \nabla v_N) \cdot w \, dx.
		\end{equation}
		We estimate \eqref{c8} term by term. Since $w$ is divergence free, we have
		\begin{equation*}
			%\label{c9}
			\begin{split}
			\int_{\Omega} \mathbb{P}(v_N \cdot \nabla \tilde{u}_N)\cdot w \, dx  &= \int_{\Omega}(v_N \cdot \nabla \tilde{u}_N)\cdot w \, dx \\
				&\leq \|\nabla \tilde{u}_N\|_{L^2(\Omega)}\|v_N\|_{L^4(\Omega)}\|w\|_{L^4(\Omega)}.
			\end{split}
		\end{equation*}
		By Sobolev inequality, $\|w\|_{L^4(\Omega)} \leq c\|w\|_{V}\leq c$ and thus
		
		\begin{equation}\label{c10}
			\begin{split}
				\|\mathbb{P}(v_N \cdot \nabla \tilde{u}_N)\|_{V'} &\leq c\|\nabla \tilde{u}_N\|_{L^2(\Omega)}\|v_N\|_{L^4(\Omega)}\\
				&\leq c\|\nabla \tilde{u}_N\|_{L^2(\Omega)}\|v_N\|_{L^2(\Omega)}^{1/2}\|\nabla v_N\|_{L^2(\Omega)}^{1/2},
			\end{split}
		\end{equation}
		where in the last line, we used the Ladyzhenskaya's inequality \cite{CF}.

		Therefore
		\begin{equation*}
			%\label{c11}
			\begin{split}
				\int_0^T \|\mathbb{P}(v_N \cdot \nabla \tilde{u}_N)\|_{V'}^2 \, dt \leq c\int_0^T \|\nabla \tilde{u}_N\|_{L^2(\Omega)}^2\|v_N\|_{L^2(\Omega)}\|\nabla v_N\|_{L^2(\Omega)} \, dt \\
				%\leq \left(  \underbrace{\int_0^T ||\nabla \tilde{u}_N||_{L^2(\Omega)}^4}_\text{$O(\frac{\varepsilon}{\sqrt{\lambda}})$}   \right)^{1/2}\left(  \underbrace{\int_0^T ||v_N||_{L^2(\Omega)}^4}_\text{$O(\varepsilon^{1/2})$} \right)^{1/4} \left(  \underbrace{\int_0^T ||\nabla v_N||_{L^2(\Omega)}^4}_\text{$O(\frac{\varepsilon}{\sqrt{\lambda}})$} \right)^{1/4}.
				\leq c\left(\int_0^T \|\nabla \tilde{u}_N\|_{L^2(\Omega)}^4 dt\right)^{1/2}\left (\int_0^T \|v_N\|_{L^2(\Omega)}^4 dt\right)^{1/4} \left(\int_0^T \|\nabla v_N\|_{L^2(\Omega)}^4 dt\right)^{1/4}.
			\end{split}
		\end{equation*}
		
		Since
		\begin{equation*}
			\lambda \| \tilde{u}_N\|_{L^4([0, T]; H^1(\Omega))}^4 \leq \varepsilon^2,
		\end{equation*}
		we have
		\begin{equation*}
			\int_0^T \|\nabla \tilde{u}_N\|_{L^2(\Omega)}^4 dt \leq \frac{\varepsilon^2}{\lambda}.
		\end{equation*}
		
		Therefore
		\begin{equation}
			\label{1est}
			\left(\int_0^T \|\nabla \tilde{u}_N\|_{L^2(\Omega)}^4 dt\right)^{1/2} \leq \frac{\varepsilon}{\sqrt{\lambda}}.
		\end{equation}
		
		Applying Lemma \ref{lemma5.2}, we have
		$$
		\left(\int_0^T \|\nabla v_N\|_{L^2(\Omega)}^4 dt \right)^{1/4}\leq c{\sqrt{\varepsilon}},\
		\left( \int_0^T \|v_N\|_{L^2(\Omega)}^4 dt \right)^{1/4}\leq c{\sqrt{\varepsilon}}.
		$$
		
		Therefore
		\begin{align}
			\label{1est3}
			\int_0^T \|\mathbb{P}(v_N \cdot \nabla \tilde{u}_N)\|_{V'}^2 \, dt\leq \frac{c\varepsilon^2}{\sqrt{\lambda}}.
		\end{align}
		
		Next, we estimate the second term of \eqref{c8}: $\displaystyle \int_{\Omega} \mathbb{P}(u_N \cdot \nabla v_N) \cdot w \, dx$.
		
		Similarly, we have
		\begin{equation*}
			\|\mathbb{P}(u_N \cdot \nabla v_N)\|_{V'} \leq \|\nabla v_N\|_{L^2(\Omega)}\|u_N\|_{L^4(\Omega)}\leq \|\nabla v_N\|_{L^2(\Omega)}\|u_N\|_{H^1(\Omega)}.
		\end{equation*}
		
		Therefore
		\begin{equation*}
			\begin{split}
				\int_0^T \|\mathbb{P}(u_N \cdot \nabla v_N)\|_{V'}^2 \, dt \leq \int_0^T \|\nabla v_N\|^2_{L^2(\Omega)}\|u_N\|^2_{H^1(\Omega)} \, dt \\
				\leq \left(\int_0^T \|\nabla v_N\|_{L^2(\Omega)}^4 dt\right)^{1/2}\left(\int_0^T \|u_N\|_{H^1(\Omega)}^4 dt\right)^{1/2}.
			\end{split}
		\end{equation*}
		
		Since
		\begin{align*}
			\left(\int_0^T \|\nabla v_N\|_{L^2(\Omega)}^4 dt \right)^{1/2}\leq c{{\varepsilon}},
		\end{align*}
		and
		%	\begin{equation*}
		%	\left(\int_0^T \|u_N\|_{H^1(\Omega)}^4 dt \right)^{1/2}\leq \left(\int_0^T \|\tilde{u}_N\|_{H^1(\Omega)}^4 dt\right)^{1/2}\leq c\left(\int_0^T \|\nabla \tilde{u}_N\|_{L^2(\Omega)}^4 dt\right)^{1/2}\leq  \frac{\varepsilon}{\sqrt{\lambda}}.
		%	\end{equation*}
		\begin{equation*}
			\left(\int_0^T \|u_N\|_{H^1(\Omega)}^4 dt \right)^{1/2}\leq \left(\int_0^T \|\tilde{u}_N\|_{H^1(\Omega)}^4 dt\right)^{1/2}\leq  \frac{\varepsilon}{\sqrt{\lambda}}.
		\end{equation*}
		
		Therefore
		\begin{align}
			\label{2est3}
			\int_0^T \|\mathbb{P}(u_N \cdot \nabla v_N)\|_{V'}^2 \, dt\leq \frac{c\varepsilon^2}{\sqrt{\lambda}}.
		\end{align}
		%The above estimates come from the Hodge-decomposition, see \cite{CF},
		%\begin{equation}\label{c12}
		%\begin{split}
		%||\nabla \cdot \tilde{u}_N||_{L^4([0, T]; L^2(\Omega))} \qquad \text{is small} \\
		%|| \tilde{u}_N ||_{L^4([0,T]; H^1(\Omega))} \leq \frac{\varepsilon}{\sqrt{\lambda}}\\
		%P : H^1(\Omega) \to H^1(\Omega) \qquad \text{is bounded.}
		%\end{split}
		%\end{equation}
		%Estimating the other term in \eqref{c8} in a similar fashion we deduce that
		
		Combining (\ref{1est3}) and (\ref{2est3}), we have
		\begin{equation}\label{c13}
			\int_0^T \|\psi\|_{V'}^2 \, dt \leq O \left(\frac{\varepsilon^2}{\sqrt{\lambda}}\right).
		\end{equation}
		%Using the Hodge-decomposition we also have
		%	Note that $\mathbb{P} \Delta v_N \in V'$ and since $\|\nabla \cdot \tilde{u}_N\|_{L^2([0, T] \times \Omega)} \leq \varepsilon$ and $\nabla \cdot {u}_N=0,$ we have
		%	\begin{align*}
		%		\|\nabla \cdot v_N\|_{L^2([0, T] \times \Omega)} \leq \varepsilon.
		%		%\int_0^T ||P\Delta v_N||_{V'} \, dt< \varepsilon.
		%	\end{align*}
		%	
		%	So
		Moreover, since
		\begin{equation}\label{c14n}
			\|\mathbb{P}(\Delta v_N)\|_{V'}\leq \|\nabla v_N\|_{L^2(\Omega)}\|w\|_{H^1(\Omega)},
		\end{equation}
		we have
		\begin{align}\label{c14}
			\int_0^T \|\mathbb{P}(\Delta v_N)\|_{V'}^2 \, dt&\leq \int_0^T \|\nabla v_N\|^2_{L^2(\Omega)}\|w\|^2_{H^1(\Omega)} \, dt \\ \nonumber
			&\leq \left(\int_0^T \|\nabla v_N\|_{L^2(\Omega)}^4 dt\right)^{1/2}\left(\int_0^T \|w\|_{H^1(\Omega)}^4 dt\right)^{1/2}\\ \nonumber
			&\leq c{{\varepsilon}}.
		\end{align}

		Denoting $\varphi := \mathbb{P} \tilde{f}+ \mathbb{P}(\Delta v_N) - \psi$, we have
		$$\frac{du_N}{dt}  + A u_N + B(u_N, u_N) = \mathbb{P}f + \varphi.$$
		Since
		%	$$\int_0^T \|\mathbb{P}((\Delta)v_N)\|_{V'}^2 \, dt\leq c\varepsilon,\ \int_0^T \|\mathbb{P} \tilde{f}\|_{V'}^2 \, dt\leq \int_0^T \|\mathbb{P} \tilde{f}\|_{L^2}^2 \, dt\leq \varepsilon^2,\ \int_0^T \|\psi\|_{V'}^2 \, dt \leq O \left(\frac{\varepsilon^2}{\sqrt{\lambda}}\right),$$
		$$\int_0^T \|\mathbb{P}(\Delta v_N)\|_{V'}^2 \, dt\leq c\varepsilon,\ \int_0^T \|\mathbb{P} \tilde{f}\|_{V'}^2 \, dt\leq \int_0^T \|\mathbb{P} \tilde{f}\|_{L^2}^2 \, dt\leq \varepsilon^2,\ \int_0^T \|\psi\|_{V'}^2 \, dt \leq O \left(\frac{\varepsilon^2}{\sqrt{\lambda}}\right),$$
		we obtain
		$$\int_0^T \|\varphi\|^2_{V'} \, dt \leq O \left(\varepsilon+\frac{\varepsilon^2}{\sqrt{\lambda}}\right).$$
		
		Since $\displaystyle \|{u}_N (x, 0) - u_0(x)\|_{L^2}^2\leq \|{\tilde{u}}_N (x, 0) - u_0(x)\|_{L^2}^2\leq {\varepsilon}^2.$ Applying Lemma \ref{lemma1}, we have
		\begin{align*}
			\sup_{[0, T]}\|u(t)-u_N(t)\|_{L^2}^2\leq O \left(\varepsilon+\frac{\varepsilon^2}{\sqrt{\lambda}}\right).
		\end{align*}
		Moreover, since
		\begin{align*}
			\left( \int_0^T \|v_N\|_{L^2(\Omega)}^4 dt\right)^{1/4}\leq c{\sqrt{\varepsilon}},
		\end{align*}
		%	we have
		%	\begin{align*}
		%	\sup_{[0, T]}\|v_N\|^2_{L^2(\Omega)}\leq O({\varepsilon}).
		%	\end{align*}
		Therefore
		\begin{align*}
			\int_0^T \|u-\tilde{u}_N\|_{L^2(\Omega)}^4 dt&\leq \int_0^T \|u-{u}_N\|_{L^2(\Omega)}^4 dt+\int_0^T \|v_N\|_{L^2(\Omega)}^4 dt\\ \nonumber
			&\leq O \left(\varepsilon^2+\frac{\varepsilon^4}{{\lambda}}\right)+O \left(\varepsilon^2\right)=O \left(\varepsilon^2+\frac{\varepsilon^4}{{\lambda}}\right).
			%	  \|u-\tilde{u}_N\|^2_{L^2(\Omega)}\leq \sup_{[0, T]}\|u(t)-u_N(t)\|^2_{L^2(\Omega)}+\sup_{[0, T]}\|v_N\|^2_{L^2(\Omega)}\leq O \left(\frac{\varepsilon^2}{\sqrt{\lambda}}\right)+O({\varepsilon})=O \left(\frac{\varepsilon^2}{\sqrt{\lambda}}\right).
		\end{align*}
		So
		\begin{align*}
			\left(\int_0^T \|u-\tilde{u}_N\|_{L^2(\Omega)}^4 dt\right)^{1/4}\leq O \left(\varepsilon^{1/2}+\frac{\varepsilon}{{\lambda^{1/4}}}\right).
		\end{align*}
		%Case(ii):
		%
		%
		%The procedure is similar to the $\tilde{u}_N |_{\partial \Omega} = 0$ case, the only difference is that in (\ref{u1}), instead of $\gamma \tilde{u}_N=0$, we have $|| \tilde{u}_N|_{\partial \Omega}||_{L^4([0, T]; H^{1/2}(\partial \Omega))}^4\leq {\varepsilon}^2.$ Therefore
		%
		%Therefore, we still have
		%$$\int_0^T ||P(v_N \cdot \nabla \tilde{u}_N)||_{V'}^2 \, dt\leq \frac{c\varepsilon^2}{\lambda},\ \int_0^T ||P(u_N \cdot \nabla v_N)||_{V'}^2 \, dt\leq \frac{c\varepsilon^2}{\lambda}.$$
		%So
		%$$\int_0^T ||\psi_{\varepsilon}||_{V'}^2 \, dt \leq O(\frac{\varepsilon^2}{\lambda}),\ \int_0^T ||\varphi_{\varepsilon}||^2_{V'} \, dt \leq c\varepsilon^2.$$
		%Following similar procedure as in Lemma \ref{lemma1}, we have
		%$$\sup_{[0, T]}\|u(t)-u_N(t)\|^2_{L^2(\Omega)}\leq O({\varepsilon}^2).$$
		%Therefore
		%\begin{equation}\label{dec10}
		%\sup_{[0, T]}||u-\tilde{u}_N||^2_{L^2(\Omega)}\leq \sup_{[0, T]}\|u(t)-u_N(t)\|^2_{L^2(\Omega)}+\sup_{[0, T]}||v_N||^2_{L^2(\Omega)}\leq O({\varepsilon}^2)+O({\varepsilon})=:\delta.
		%\end{equation}
	\end{proof}
	%Note, in the above proof, we have used the property that $\displaystyle \lambda \sim O({\varepsilon}^2)$.
	%\textbf{Stability of the scheme}
	\begin{lemma} \label{lemforuN}
%		Given $\varepsilon_1,\ \varepsilon_2,\ \varepsilon_3,\ \varepsilon_4> 0$, assume that $(u, p)$ satisfies (\ref{c2}), then, there exists $(\tilde{u}_N, \tilde{p}_N) \in \mathscr{F}_N$ satisfying %$(u_\mathscr{F}, p_\mathscr{F}) \in \mathscr{F}_N$ with $|| u-u_\mathscr{F} ||_{C^{1,2}((0, T)\times \Omega)}\leq \varepsilon$ and $|| p-p_\mathscr{F} ||_{C^{0,1}((0, T)\times \Omega)}\leq \varepsilon$,
%			\begin{align} \label{condi1}
%		 &\sup_{t\in [0,T]}\|u-\tilde{u}_N\|_{L^2 (\Omega)}\leq \varepsilon_1, \|u-\tilde{u}_N\|_{H^{1, 2} (\Omega\times [0, T] )}\leq \varepsilon_2,\\ \nonumber &\left(\int^T_0\|u-\tilde{u}_N\|^4_{W^{1,4}(\Omega)}dt\right)^{1/4}\leq \varepsilon_3, \|p-\tilde{p}_N\|_{L^2 ([0,T];H^1(\Omega))}\leq \varepsilon_4.
	Given $\varepsilon> 0$, assume that $(u, p)$ satisfies (\ref{c2}), then, there exists $(\tilde{u}_N, \tilde{p}_N) \in \mathscr{F}_N$ satisfying %$(u_\mathscr{F}, p_\mathscr{F}) \in \mathscr{F}_N$ with $|| u-u_\mathscr{F} ||_{C^{1,2}((0, T)\times \Omega)}\leq \varepsilon$ and $|| p-p_\mathscr{F} ||_{C^{0,1}((0, T)\times \Omega)}\leq \varepsilon$,
\begin{align} \label{condi1}
&\sup_{t\in [0,T]}\|u-\tilde{u}_N\|_{L^2 (\Omega)}\leq \varepsilon,\ \|u-\tilde{u}_N\|_{H^{1, 2} (\Omega\times [0, T] )}\leq \varepsilon,\\ \nonumber &\left(\int^T_0\|u-\tilde{u}_N\|^4_{W^{1,4}(\Omega)}dt\right)^{1/4}\leq \varepsilon,\ \|p-\tilde{p}_N\|_{L^2 ([0,T];H^1(\Omega))}\leq \varepsilon.
		 	\end{align}
		\end{lemma}
\begin{proof}
%Here, we need to show two parts: (i) the solution $u$ of the NSEs can satisfy those conditions; (ii) the approximate solution $\tilde{u}_N$ can satisfy those conditions. 
From Theorem \ref{theorappr}, as long as the solution $(u, p)$ of the NSEs belongs to the spaces in (\ref{condi1}), we can find $(\tilde{u}_N, \tilde{p}_N) \in \mathscr{F}_N$ as smooth as we want and close to $(u, p)$, which means (\ref{condi1}) holds. From the classical results of the 2-D NSEs, we know that we can find the solution $(u, p)$ that belongs to the spaces in (\ref{condi1}).
\end{proof}
	\textbf{Proof of Theorem \ref{thmay6}:}
	\begin{proof}
		%Next, we consider the function:
		%\begin{equation*}
		%\begin{split}
		%|| \tilde{u}_N (x, 0) - u_0(x)||_{L^2(\Omega)}^2 +|| \partial_t \tilde{u}_N - \Delta \tilde{u}_N + \tilde{u}_N\cdot \nabla \tilde{u}_N + \nabla p_\mathscr{F} - f||_{L^2([0, T] \times \Omega)}^2 \\
		%+ || \nabla \cdot \tilde{u}_N ||_{L^2([0, T] \times \Omega)}^2
		%+ \lambda || \tilde{u}_N||_{L^4([0, T]; H^1(\Omega))}^4.
		%\end{split}
		%\end{equation*}
		From Lemma {\ref{lemforuN}}, 
		given $\varepsilon> 0$, assume that $u$ is a strong solution of (\ref{c2}) and there is an $N$ such that $\tilde{u}_N \in \mathscr{F}_N$ satisfying %$(u_\mathscr{F}, p_\mathscr{F}) \in \mathscr{F}_N$ with $|| u-u_\mathscr{F} ||_{C^{1,2}((0, T)\times \Omega)}\leq \varepsilon$ and $|| p-p_\mathscr{F} ||_{C^{0,1}((0, T)\times \Omega)}\leq \varepsilon$,
		$\displaystyle \sup_{t\in [0,T]}\|u-\tilde{u}_N\|_{L^2 (\Omega)}\leq \varepsilon$, $\displaystyle \|u-\tilde{u}_N\|_{H^{1, 2} (\Omega\times [0, T] )}\leq \varepsilon$, $\displaystyle \left(\int^T_0\|u-\tilde{u}_N\|^4_{W^{1,4}(\Omega)}dt\right)^{1/4}\leq \varepsilon$, $\displaystyle \|p-\tilde{p}_N\|_{L^2 ([0,T];H^1(\Omega))}\leq \varepsilon$. Now, we estimate the left hand side of (\ref{dec8}) term by term:
		%	\begin{align}
		%	\label{dec1}
		%	&\| \tilde{u}_N (x, 0) - u_0(x)\|_{L^2(\Omega)}^2\\ \nonumber
		%	&= \int_{\Omega} \left(\tilde{u}_N (x, 0) - u_0(x) \right)^2 dx\\ \nonumber
		%	&=  \int_{\Omega} \left(-\int^{T/2}_{0} \partial_t \tilde{u}_N (x, t)dt +\tilde{u}_N \left(x, \frac{T}{2}\right)+\int^{T/2}_{0} \partial_t u (x, t)dt -u \left(x, \frac{T}{2}\right)\right)^2 dx\\ \nonumber
		%	&=  \int_{\Omega} \left(\int^{T/2}_{0} \left(\partial_t u (x, t)-\partial_t \tilde{u}_N (x, t)\right)dt +\left(\tilde{u}_N \left(x, \frac{T}{2}\right)-u \left(x, \frac{T}{2}\right)\right)\right)^2 dx\\ \nonumber
		%	&\leq  c \int_{\Omega} \left(\int^{T/2}_{0} \left(\partial_t u (x, t)-\partial_t \tilde{u}_N (x, t)\right)dt \right)^2+\left(\tilde{u}_N \left(x, \frac{T}{2}\right)-u \left(x, \frac{T}{2}\right)\right)^2 dx
		%	\end{align}
		
		%	Since $\displaystyle \sup_{t\in [0,T]}\|u-\tilde{u}_N\|_{L^2 (\Omega)}<\varepsilon_1$, i.e., $\displaystyle \sup_{t\in [0,T]} \left(\int_{\Omega} (u-\tilde{u}_N)^2 dx \right)^{1/2}<\varepsilon_1$, \\
		%	so $\displaystyle \int_{\Omega} \left(u\left(x, \frac{T}{2}\right)-\tilde{u}_N \left(x, \frac{T}{2}\right)\right)^2 dx<\varepsilon^2_1$.\\
		%$\displaystyle \sup \left|\tilde{u}_N \left(x, \frac{T}{2}\right)-u \left(x, \frac{T}{2}\right)\right|\leq \varepsilon$ and $\displaystyle \int^{T/2}_{0} \left(\partial_t u (x, t)-\partial_t \tilde{u}_N (x, t)\right)dt \leq \varepsilon$, we have
		%\begin{align}
		%\label{dec2}
		%\| \tilde{u}_N (x, 0) - u_0(x)\|_{L^2(\Omega)}^2\leq c_\Omega \left(\varepsilon \cdot \frac{T}{2}+\varepsilon\right)^2.
		%\end{align}
		
		Since $\displaystyle \sup_{t\in [0,T]}\|u-\tilde{u}_N\|_{L^2 (\Omega)}\leq \varepsilon$, we have
		\begin{align}
			\label{dec5}
			\|u(x,0)-\tilde{u}_N (x,0)\|^2_{L^2 (\Omega)}\leq \varepsilon^2.
		\end{align}
		
		Since $\nabla \cdot u=0$, we have
		\begin{align}
			\label{dec6}
			\| \nabla \cdot \tilde{u}_N \|_{L^4([0, T] ; L^2(\Omega))}^4&=\| \nabla \cdot \tilde{u}_N-\nabla \cdot u \|_{L^4([0, T] ; L^2(\Omega))}^4\\ \nonumber
			&\leq c\| \tilde{u}_N-u \|_{L^4([0, T] ; H^1(\Omega))}^4\\ \nonumber
			&\leq c\varepsilon^4.
		\end{align}
		
		Consider
		\begin{align*}
			\lambda \| \tilde{u}_N\|_{L^4([0, T]; H^1(\Omega))}^4&\leq c\lambda\| \tilde{u}_N-u\|_{L^4([0, T]; H^1(\Omega))}^4+c\lambda\|u\|_{L^4([0, T]; H^1(\Omega))}^4\\ \nonumber
			&\leq c \lambda \varepsilon^4+c\lambda F(G).
			%&\leq c \lambda \cdot T \cdot \left(\sup || \nabla \tilde{u}_N-\nabla u||_{L^2}^4+\sup || \nabla u||_{L^2}^4 \right).
		\end{align*}
		
		%We have $\displaystyle \int^{T}_{0} \| \nabla \tilde{u}_N-\nabla u\|_{L^2}^4dt\leq c \varepsilon^4_2$ and $\| \nabla u\|_{L^2}^4 \leq M_G$, where $M_G$ is a constant relates to the Grashof number $G$.
		
		Picking $\lambda$ small enough, we have
		\begin{align}
			\label{dec7}
			\lambda \| \tilde{u}_N\|_{L^4([0, T]; H^1(\Omega))}^4 \leq c\varepsilon^4.
		\end{align}
		
		Consider
		\begin{align}
			\label{dec3}
			&\|\partial_t \tilde{u}_N - \Delta \tilde{u}_N + \tilde{u}_N\cdot \nabla \tilde{u}_N + \nabla \tilde{p}_N - f\|_{L^2(\Omega\times [0, T] )}^2\\ \nonumber
			&=\|\partial_t \tilde{u}_N - \Delta \tilde{u}_N + \tilde{u}_N\cdot \nabla \tilde{u}_N + \nabla \tilde{p}_N -\left(\partial_t u - \Delta u + u \cdot \nabla u + \nabla p\right)\|_{L^2(\Omega\times [0, T] )}^2\\ \nonumber
			&=\|(\partial_t \tilde{u}_N - \partial_t u)+(\Delta u - \Delta \tilde{u}_N) + (\tilde{u}_N\cdot \nabla \tilde{u}_N-u \cdot \nabla u) + \left(\nabla \tilde{p}_N - \nabla p)\right)\|_{L^2(\Omega\times [0, T] )}^2\\ \nonumber
			&\leq c\|\partial_t \tilde{u}_N - \partial_t u\|_{L^2(\Omega\times [0, T] )}^2+c\| \Delta u - \Delta \tilde{u}_N\|_{L^2(\Omega\times [0, T] )}^2\\ \nonumber
			&+c\|\tilde{u}_N\cdot \nabla \tilde{u}_N-u \cdot \nabla u\|_{L^2(\Omega\times [0, T] )}^2+c\| \nabla \tilde{p}_N - \nabla p\|_{L^2(\Omega\times [0, T] )}^2.
		\end{align}
		
		We have $$\displaystyle \|\partial_t \tilde{u}_N - \partial_t u\|_{L^2(\Omega\times [0, T] )}^2\leq  \varepsilon^2,$$
		$$\displaystyle \| \Delta u - \Delta \tilde{u}_N\|_{L^2(\Omega\times [0, T] )}^2\leq \varepsilon^2,$$ and $$\displaystyle \| \nabla \tilde{p}_N - \nabla p\|_{L^2(\Omega\times [0, T] )}^2\leq \varepsilon^2.$$
		Moreover
		\begin{align}
			\label{dec4}
			&\| \tilde{u}_N\cdot \nabla \tilde{u}_N-u \cdot \nabla u\|_{L^2(\Omega\times [0, T] )}^2\\ \nonumber
			&=\| \tilde{u}_N\cdot \nabla \tilde{u}_N-u\cdot \nabla \tilde{u}_N+u\cdot \nabla \tilde{u}_N-u \cdot \nabla u\|_{L^2(\Omega\times [0, T] )}^2\\ \nonumber
			&=\| (\tilde{u}_N-u)\cdot \nabla \tilde{u}_N+u\cdot (\nabla \tilde{u}_N-\nabla u)\|_{L^2(\Omega\times [0, T] )}^2\\ \nonumber
			&\leq \|\tilde{u}_N-u\|^2_{L^4(\Omega\times [0, T] )} \| \nabla \tilde{u}_N\|^2_{L^4(\Omega\times [0, T] )}+\|u\|^2_{L^4(\Omega\times [0, T] )} \|\nabla \tilde{u}_N-\nabla u\|^2_{L^4(\Omega\times [0, T] )}.
		\end{align}
		%	From (\ref{dec7}),
		We have
		$$\displaystyle \| \nabla \tilde{u}_N\|_{L^4(\Omega\times [0, T] )}\leq \| \nabla \tilde{u}_N-\nabla u\|_{L^4(\Omega\times [0, T] )}+\| \nabla  u\|_{L^4(\Omega\times [0, T] )}\leq c\varepsilon+cF(G)\leq cF(G).$$ Moreover, $$\|\tilde{u}_N-u\|_{L^4(\Omega\times [0, T] )}\leq c \varepsilon,\ \|u\|_{L^4(\Omega\times [0, T] )} \leq c F(G),\ \|\nabla \tilde{u}_N-\nabla u\|_{L^4(\Omega\times [0, T] )} \leq c \varepsilon.$$So
		$$\| \tilde{u}_N\cdot \nabla \tilde{u}_N-u \cdot \nabla u\|_{L^2(\Omega\times [0, T] )}^2\leq c \varepsilon^2.$$
		Therefore
		\begin{align}
			\label{dec21}
			\|\partial_t \tilde{u}_N - \Delta \tilde{u}_N + \tilde{u}_N\cdot \nabla \tilde{u}_N + \nabla \tilde{p}_N - f\|_{L^2(\Omega\times [0, T] )}^2 \leq O\left(\varepsilon^2\right).
		\end{align}
		Moreover
		\begin{align}
			\label{dec9}
			\| \tilde{u}_N|_{\partial \Omega}\|_{L^4([0, T]; H^{1/2}(\partial \Omega))}^4 &=\| \tilde{u}_N|_{\partial \Omega}-u|_{\partial \Omega}\|_{L^4([0, T]; H^{1/2}(\partial \Omega))}^4 \\ \nonumber
			&\leq c \| \tilde{u}_N-u\|_{L^4([0, T]; H^{1}(\Omega))}^4\\ \nonumber
			&\leq c \varepsilon^4.
		\end{align}
		
		In summary, combining (\ref{dec5})-(\ref{dec9}), we have
		\begin{align*}
			%\label{dec8}
			%\begin{split}
			&	\| \tilde{u}_N|_{\partial \Omega}\|_{L^4([0, T]; H^{1/2}(\partial \Omega))}^4+\| \tilde{u}_N (x, 0) - u_0(x)\|_{L^2(\Omega)}^2 \\ \nonumber
			&	+\| \partial_t \tilde{u}_N - \Delta \tilde{u}_N + \tilde{u}_N\cdot \nabla \tilde{u}_N + \nabla \tilde{p}_N - f\|_{L^2(\Omega\times [0, T] )}^2 \\ \nonumber
			&	+ \| \nabla \cdot \tilde{u}_N \|_{L^4([0, T] ; L^2(\Omega))}^4
			+ \lambda \| \tilde{u}_N\|_{L^4([0, T]; H^1(\Omega))}^4\leq O\left(\varepsilon^2\right).
			%\end{split}
		\end{align*}
%		Given any $\varepsilon>0$, taking $\varepsilon_1=\varepsilon_2=\varepsilon_3=\varepsilon_4=\frac{\sqrt{\varepsilon}}{2}$, we have $O\left(\varepsilon^2_1+\varepsilon^2_2+ \varepsilon^2_3+\varepsilon^2_4\right)=O\left(\varepsilon\right).$
	\end{proof}
\textbf{Proof of Theorem \ref{thmay4}:}
\begin{proof}
	Considering
	\begin{align*}
	&\|\tilde u_{N_1}-\tilde u_{N_2}\|_{L^4([0, T]; L^2(\Omega))}\\
	&=\|\tilde u_{N_1}-u_1+u_2-\tilde u_{N_2}+u_1-u_2\|_{L^4([0, T]; L^2(\Omega))}\\ \nonumber
	&\leq \|\tilde u_{N_1}-u_1\|_{L^4([0, T]; L^2(\Omega))}+\|u_2-\tilde u_{N_2}\|_{L^4([0, T]; L^2(\Omega))}+\|u_1-u_2\|_{L^4([0, T]; L^2(\Omega))}.
	\end{align*}
	
	From (\ref{diff1}), $\displaystyle \|\tilde u_{N_1}-u_1\|_{L^4([0, T]; L^2(\Omega))}\leq O \left(\varepsilon^{1/2}+\frac{\varepsilon}{{\lambda}^{1/4}}\right)$ and $\displaystyle \|u_2-\tilde u_{N_2}\|_{L^4([0, T]; L^2(\Omega))}\leq O \left(\varepsilon^{1/2}+\frac{\varepsilon}{{\lambda}^{1/4}}\right)$, from the stability of solutions of NSE, we have $\|u_1-u_2\|_{L^4([0, T]; L^2(\Omega))}\leq O(\|u_{0, 1}-u_{0, 2}\|_{L^2(\Omega)}+\|f_1 -f_2\|_{L^4([0, T]; L^2(\Omega))})$.
	
	Therefore $$\|u_{N_1}-u_{N_2}\|_{L^4([0, T]; L^2(\Omega))}\leq O\left(\varepsilon^{1/2}+\frac{\varepsilon}{{\lambda}^{1/4}}+\|u_{0, 1}-u_{0, 2}\|_{L^2(\Omega)}+\|f_1 -f_2\|_{L^4([0, T]; L^2(\Omega))}\right).$$
	This implies that our scheme is approximately stable.
\end{proof}
	\section*{Acknowledgement}
	 J. Tian's work is supported in part by the AMS Simons Travel Grant.

\end{document}